\documentclass[a4paper,reqno]{amsart}

\usepackage{authblk}
\usepackage[french,british]{babel}
\usepackage{amssymb,amsmath,amsthm,amscd,mathrsfs,graphicx,color,multicol,booktabs,bm,manfnt}
\usepackage[cmtip,all,matrix,arrow,tips,curve]{xy}
\usepackage[foot]{amsaddr}

\usepackage{mathabx}

\usepackage{mathrsfs}
\usepackage[colorlinks=true,citecolor=blue, urlcolor=blue, linkcolor=blue, pagebackref]{hyperref}

\usepackage{manfnt}
\usepackage{centernot}   

\newcommand{\blambda}{{\boldsymbol\lambda}}

\newcommand{\bmu}{{\boldsymbol\mu}}

\newcommand{\CC}{{\mathbb C}}
\newcommand{\cA}{{\mathscr A}}

\newcommand{\cD}{{\mathscr D}}
\newcommand{\cE}{{\mathscr E}}
\newcommand{\cF}{{\mathscr F}}
\newcommand{\cG}{{\mathscr G}}
\newcommand{\cH}{{\mathscr H}}

\newcommand{\cL}{{\mathscr L}}
\newcommand{\cM}{{\mathscr M}}

\newcommand{\cO}{{\mathscr O}}
\newcommand{\cP}{{\mathscr P}}

\newcommand{\cS}{{\mathscr S}}

\newcommand{\cX}{{\mathscr X}}

\newcommand{\dd}{{\textbf d}}

\newcommand{\dra}{\dashrightarrow}

\newcommand{\gh}{\mathfrak{h}}

\newcommand{\gM}{\mathfrak{M}}

\newcommand{\hra}{\hookrightarrow}

\newcommand{\la}{\langle}

\newcommand{\lra}{\longrightarrow}

\newcommand{\ov}{\overline}
\newcommand{\PP}{{\mathbb P}}
\newcommand{\QQ}{{\mathbb Q}}
\newcommand{\ra}{\rangle}
\newcommand{\RR}{{\mathbb R}}

\newcommand{\sF}{{\mathsf F}}

\newcommand{\sM}{{\mathsf M}}

\newcommand{\wh}{\widehat}
\newcommand{\wt}{\widetilde}
\newcommand{\ZZ}{{\mathbb Z}}

\theoremstyle{plain}

\newtheorem{thm}{Theorem}[section]
\newtheorem*{thm*}{Theorem}

\newtheorem{crl}[thm]{Corollary}
\newtheorem*{hyp*}{Hypothesis}
\newtheorem{lmm}[thm]{Lemma}
\newtheorem{prp}[thm]{Proposition}
\newtheorem{prp-dfn}[thm]{Proposition-Definition}

\theoremstyle{definition}

\newtheorem{dfn}[thm]{Definition}

\theoremstyle{remark}

\newtheorem{expl}[thm]{Example}

\newtheorem*{qst*}{Main Question}
\newtheorem{rmk}[thm]{Remark}

\DeclareMathOperator{\alb}{alb}
\DeclareMathOperator{\Alb}{Alb}

\DeclareMathOperator{\Aut}{Aut}

\DeclareMathOperator{\Blow}{Bl}
\DeclareMathOperator{\ch}{ch}
\DeclareMathOperator{\CH}{CH}

\DeclareMathOperator{\cl}{cl}

\DeclareMathOperator{\divisore}{div}

\DeclareMathOperator{\GL}{GL}

\DeclareMathOperator{\Id}{Id}
\DeclareMathOperator{\im}{Im}

\DeclareMathOperator{\Kum}{Kum}

\DeclareMathOperator{\Nm}{Nm}
\DeclareMathOperator{\NS}{NS}

\DeclareMathOperator{\OG}{OG6}
\DeclareMathOperator{\OGten}{OG10}
\DeclareMathOperator{\ord}{ord}
\DeclareMathOperator{\Ort}{O}

\DeclareMathOperator{\Pic}{Pic}

\DeclareMathOperator{\Prym}{Pr}

\DeclareMathOperator{\Supp}{Supp}

\DeclareMathOperator{\Trans}{T}

\DeclareMathOperator{\WIT}{WIT}

\usepackage{multirow}

 \makeindex
\begin{document}
\title{Theta groups and projective models of  hyperk\"ahler varieties}
\author{Kieran G. O'Grady}
\address{Dipartimento di Matematica Guido Castelnuovo, Sapienza Universit\`a di Roma, Piazzale A.~Moro 5, 00185 Roma, Italia}
\email{ogrady@mat.uniroma1.it}
\date{\today}
\thanks{Partially supported by PRIN 2017YRA3LK}
\dedicatory{Dedicato alla memoria di Alberto Collino}
\begin{abstract}
We define the theta  group associated to a simple coherent sheaf $\cF$ on a hyperk\"ahler manifold $X$ of Kummer type or OG6 type, provided $g^{*}(\cF)$ is isomorphic to $\cF$ for every automorphism $g$ of $X$ acting trivially on $H^2(X)$. Note that this condition is satisfied if $\cF$ is invertible, if $\cF$  is one of the rank $4$ stable vector bundles on general polarized HK fourfolds  with certain discrete invariants constructed
 in~\cite{og:modonkum}, or if $\cF$ is the tangent bundle.  We compute the commutator pairings of  theta groups of  line bundles and the rank $4$ modular vector bundles of~\cite{og:modonkum} (the commutator pairing of the tangent bundle is trivial).   We have been  motivated by the quest for an explicit description of locally complete families of polarized varieties of Kummer (or OG6) type.
\end{abstract}
  \maketitle
\bibliographystyle{amsalpha}
\section{Introduction}\label{sec:intro}
\setcounter{equation}{0}
\subsection{Background and motivation}
\setcounter{equation}{0}
Hyperk\"ahler (HK) manifolds are similar to compact complex tori   in many respects. The theta group of a line bundle  on an abelian variety  plays a key r\^ole in the analysis of projective models of such varieties, see~\cite{mum-eq-abvars}. In this paper we define and study an analogue of the theta group for HK manifolds of Kummer and OG6 type. Let  $X$ be 
a HK manifold. The normal subgroup $\Aut^0(X)\lhd\Aut(X)$   of automorphisms acting trivially on $H^2(X)$ depends only on  the deformation class of $X$, and it has been determined for the known deformation classes. If $X$ is of type 
$K3^{[n]}$ or of type OG10 then $\Aut^0(X)$  is trivial, 
if $X$ is  of type $\Kum_n$ then $\Aut^0(X)$ is  the semidirect product 
$\ZZ/(2)\ltimes (\ZZ/(n+1))^4$ where $\ZZ/(2)$ acts on  $\ZZ/(n+1)$ via multiplication by $-1$, and if  $X$ is of type OG6 then $\Aut^0(X)$ is $(\ZZ/2)^8$.
Our idea is to define the theta group by replacing the group of translations of a complex torus with the largest abelian subgroup    $T(X)<\Aut^0(X)$, i.e.(forgetting the deformation classes with trivial $\Aut^0$) $(\ZZ/(n+1))^4$ if $X$ is  of type $\Kum_n$, and  $(\ZZ/2)^8$  if  $X$ is of type OG6. 
Let $\cF$ be a (coherent) sheaf on $X$ such that
\begin{equation}\label{moreau}
g^{*}(\cF)\cong\cF\quad \forall g\in T(X).
\end{equation}
Note that~\eqref{moreau} holds if $\cF$ is an invertible sheaf, if $\cF$  is one of the rank $4$ stable vector bundles on general polarized HK fourfolds  with certain discrete invariants constructed
 in~\cite{og:modonkum}, or if $\cF$ is the tangent bundle.
One may mimic Mumford's definition of theta group $\cG(\cF)$ associated to $\cF$. In particular  
$\cG(\cF)$ is a $\CC^{*}$ extension of $T(X)$ and there is an associated commutator pairing  
 $T(X)\times T(X)\to \CC^{*}$.  If the commutator pairing is non degenerate the theta group is isomorphic to a Heisenberg group, and it is well-known that the representations of such groups are severely restricted (Stone - von Neumann). Since 
 $\cG(\cF)$ acts naturally on the space of sections of $\cF$, the upshot is that if  the commutator pairing is non degenerate then the representation space $H^0(X,\cF)$ is fully determined by its dimension.

The present paper has been motivated by the quest for an explicit description of locally complete families of polarized varieties of Kummer (or OG6) type. Because of the analogy with  abelian varieties, we believe that an understanding of theta groups will be instrumental in producing such families. More precisely, the philosophy that emerges from the present work is that if the commutator pairing of a  sheaf $\cF$ on a HK variety $X$ is non degenerate and the  space of global sections is non trivial (but not too big), then it should be possible to give an explicit description of the general deformation of the pair $(X,\cF)$.

The recent paper~\cite{floccari-kum3} contains results which are related to the ideas in the present work. 

\subsection{Main results}
\setcounter{equation}{0}
For the precise definition of the theta group and related notions we refer to Section~\ref{sec:gruppiteta}. Below are simplified versions of our main results on the theta group of a line bundle 
on a HK manifold of Kummer or OG6 type. More detailed  versions are given in Theorems~\ref{thm:compairkum} and~\ref{thm:compairog6}. Before stating the results we recall that if $X$ is a HK manifold and $\alpha\in H^2(X;\ZZ)$, then the \emph{divisibility of $\alpha$} is given by the non negative generator of the ideal 
$\{q_X(\alpha,\beta)\mid\beta\in H^2(X;\ZZ)\}$ (here $q_X$ is the Beauville-Bogomolov-Fuiki (BBF) symmetric bilinear form of $X$). We denote the divisibility of $\alpha$ by $\divisore(\alpha)$. If $L$ is a line bundle on $X$ we let 
$\divisore(L)=\divisore(c_1(L))$. 
\begin{thm}\label{thm:heisenkum}
Let $X$ be a $2n$ dimensional HK manifold of Kummer type, and let $L$ be a primitive line bundle on $X$. 
The theta group of  $L$ is a Heisenberg group if and only if the following two conditions are satisfied:
\begin{enumerate}
\item 
$\divisore(L)=2$ and $n$ is even, or $\divisore(L)=1$ (no restriction on $n$ in this case);
\item 
$\gcd\left\{n+1,\frac{q_X(L)}{2}\right\}=1$ (recall that  $q_X$ is even).
\end{enumerate}
\end{thm}
\begin{thm}\label{thm:heisenog6}
Let $X$ be a HK manifold of type $\OG$, and let $L$ be a primitive line bundle on $X$.   
The theta group of  $L$ is a Heisenberg group if and only if
 $\divisore(L)=1$ and $q_X(L)$ is not divisible by $4$.
\end{thm}
Let  $e$ be a positive integer such that $e\equiv -6 \pmod{16}$, and  
let $(M,h)$ be a general polarized HK fourfold of Kummer type such that $q_M(h)=e$ and the divisibility of $h$ is $2$. 
In~\cite{og:modonkum} we have shown that there exists a slope stable rank $4$ vector bundle $\cF$ on $M$ such that
\begin{equation}
\det\cF\cong\cO_M(h),\quad \Delta(\cF):=8c_2(\cF)-3c_1(\cF)^2=c_2(M).
\end{equation}
We have also proved that $g^{*}(\cF)\cong\cF$ for every $g\in\Aut^0(X)$, and
hence the theta group $\cG(\cF)$ is defined.
\begin{thm}\label{thm:heisenrg4}
Keeping notation as above, the theta group of  $\cF$ is a Heisenberg group if and only if
 $q_M(h)$ is not divisible by $3$.
\end{thm}
Lastly, we remark that the commutator pairing of the tangent bundle of a HK manifold of Kummer type or of type OG6 is trivial, see Example~\ref{expl:tancom}.
\subsection{Outline of the paper}
\setcounter{equation}{0}
In Section~\ref{sec:gruppiteta} we give the details of the definition of the theta group $\cG(\cF)$, we recall the definition of the Heisenberg representation, and we discuss the representation of $\cG(\cF)$ on the space of global sections of $\cF$.

Section~\ref{sec:compkum} is devoted to the computation of the commutator pairing of line bundles on HK manifolds of Kummer type. The main result  is Theorem~\ref{thm:compairkum}, which is a more precise version of Theorem~\ref{thm:heisenkum}.  The proof goes roughly as follows. One may reduce to the case of a generalized Kummer because the commutator pairing is  invariant under deformation and  under birational maps. For line bundles on  a generalized Kummer $K_n(A)$ associated to an abelian surface $A$ one has to treat two cases: a line bundle \lq\lq coming \rq\rq\ from $A$, and the  square root of the line bundle associated to the divisor $\Delta_n(A)$ parametrizing non reduced subschemes of $A$. If a line bundle \lq\lq comes\rq\rq\ from $A$, i.e.~it is equal to $\mu_n(\ell)$ where $\ell$  is a line bundle on $A$ (see Subsection~\ref{subsec:recaponkumm}  for the definition of the map 
$\mu_n\colon H^2(A)\to H^2(K_n(A))$ - since $K_n(A)$ is regular  we denote by the same symbol the isomorphism class of a line bundle and its first Chern class), one may deform $A$ to a product 
$C_1\times C_2$ of elliptic curves so that $\ell=\cO_{C_1}(D_1)\boxtimes\cO_{C_2}(D_2)$. It follows that it suffices to compute the commutator pairing of $\mu_n(\cO_{C_i}(D_i))$ for $i\in\{1,2\}$.
There are two lagrangian fibrations $\pi_i\colon K_n(C_1\times C_2)\to\PP^n$, and  $\mu_n(\cO_{C_i}(D_i)$ is a multiple of $\pi_i^{*}(\cO_{\PP^n}(1))$. Thus we are reduced to computing the commutator pairing of $\pi_i^{*}(\cO_{\PP^n}(1))$. The space of global sections of the latter is identified with the space of global sections of $\cO_{C_i}((n+1)p_i)$, where $p_i\in C_i$ is the zero of the addition law, and  hence is the Heisenberg representation of the theta group of the line bundle  
$\cO_{C_i}((n+1)p_i)$ on the elliptic curve $C_i$.
From this one gets the commutator pairing of $\pi_i^{*}(\cO_{\PP^n}(1))$.

Lastly, if $L^{\otimes 2}\cong\cO_{K_n(A)}(\Delta_n(A))$, one proves that the commutator pairing of $L$ is trivial by lifting the action of $T(K_n(A))$ on $K_n(A)$ to an action on the double cover  $\wt{K}_n(A)$ of $K_n(A)$ ramified over 
$\Delta_n(A)$.

In Section~\ref{sec:contocomm} we compute  the commutator pairing of line bundles on HK manifolds of type OG6. The main result  is Theorem~\ref{thm:compairog6}, which is a more precise version of Theorem~\ref{thm:heisenog6}.  The proof is a more intricate version of the proof of the main result of Section~\ref{sec:compkum}. 
By deformation it suffices  to compute  the commutator pairing of line bundles on two models of HK varieties of type OG6 which have been previously studied, namely the symplectic desingularizations 
 $\wt{K}_v(J)$ of an Albanese fiber of the moduli spaces of sheaves on a $2$-dimensional Jacobian $J$ with Mukai vectors $v=(0,2h,-2)$ and $v=(0,2h,0)$, where $h$ is the principal polarization of $J$. As in the previous case, the non-trivial contributions to the commutator pairing come from spaces of global sections of \lq\lq lagrangian line bundles\rq\rq.

Section~\ref{sec:rango4} contains the proof of Theorem~\ref{thm:heisenrg4}, which is easy once one has Theorem~\ref{thm:heisenkum}. 
\subsection{Acknowledgments}
I would like to thank Bert van Geemen for the interest he took in the arguments of this paper. Thanks go to the referee for a very careful reading of the paper, in particular for pointing out a computational blunder.

\section{Theta groups}\label{sec:gruppiteta}
\subsection{Automorphisms of very general HK manifolds}\label{subsec:autoverygen}
\setcounter{equation}{0}
Let $X$ be a HK manifold. Let $\Aut^0(X)<\Aut(X)$ be the normal subgroup of automorphisms acting trivially on $H^2(X)$. If $Y$ is a HK manifold deformation equivalent to $X$ then $\Aut^0(Y)$ is isomorphic to $\Aut^0(X)$, see  Theorem 2.1 in~\cite{hass-tschink-lag-planes}. 

The  HK manifold $X$ is  of \emph{type $K3^{[n]}$} if  it is a deformation of the Hilbert scheme (or Douady space) parametrizing length $n$ subschemes of a $K3$ surface,  it is of
\emph{type $Kum_n$} (here $n\ge 2$) if it is a deformation of the generalized Kummer manifold $K_n(A)\subset A^{[n+1]}$, where $A$ is a compact complex torus of dimension $2$, see~\cite{beaucy}. Lastly $X$ is of \emph{type $\OGten$} or of \emph{type $\OG$} if it is a deformation of  the symplectic desingularization of the $10$ dimensional moduli space of semistable sheaves on a $K3$ surface constructed in~\cite{og10}, respectively the Albanese fiber of the symplectic desingularization of the $10$ dimensional moduli space of semistable sheaves on an abelian surface constructed in~\cite{og6}.
\begin{expl}\label{expl:eccaut}
\begin{enumerate}
\item
If $X$ is of type $K3^{[n]}$ or OG10, then $\Aut^0(X)$ is trivial, see Proposition 10 in~\cite{beau-katata} and Theorem 3.1 in~\cite{mon-wand} respectively. 
\item
 Let $A$ be a compact complex torus of dimension $2$, and let $G_A$ be the subgroup of the group of automorphisms of $A$ (as complex manifold, we forget the group law) generated by multiplication by $(-1)$ and translations $x\mapsto x+\tau$ where $\tau\in A[n+1]$. If $g\in G_A$ and $[Z]\in K_n(A)$ then $g(Z)$ is a subscheme parametrized by $K_n(A)$, and hence we get an inclusion $G_A<\Aut(K_n(A))$. In fact $G_A=\Aut^0(K_n(A))$, see Corollary 5 in~\cite{boiss-high-enriques}. Note that $G_A$ is isomorphic to the semidirect product $\ZZ/(2)\ltimes \ZZ/(n+1)^4$ where $\ZZ/(2)$ acts on  $\ZZ/(n+1)^4$ via multiplication by $-1$.
\item
If $X$ is of type  OG6, then  $\Aut^0(X)$ is  isomorphic to $(\ZZ/(2))^8$, see Theorem 5.2 in~\cite{mon-wand}. 
\end{enumerate}
\end{expl}
\begin{prp}
Let $X$ be a HK manifold of type $\Kum_n$ for $n\ge 2$. There is a unique abelian subgroup of $\Aut^0(X)$ of index $2$, and it is isomorphic to $(\ZZ/(n+1))^4$.  
\end{prp}
\begin{proof}
We may assume that $X=K_n(A)$. Then $\Aut^0(K_n(A))$ is isomorphic to the group $G_A<\Aut(A)$ described above. The normal  subgroup  $A[n+1]\lhd G_A$ is abelian of index $2$, and is isomorphic to $(\ZZ/(n+1))^4$. Suppose that $H< G_A$ is a different abelian subgroup of  index $2$. Then $H\cap A[n+1]$ has index $2$ in $A[n+1]$. Let  $t\in H\cap  A[n+1]$ and let  $g\in(H\setminus A[n+1])$; the equality $gt=tg$ gives that $2t=0$. Hence $H\cap A[n+1]\subset A[2]$, and this is a contradiction 
because $ A[n+1]\cap A[2]$ has index greater than $2$ in $ A[n+1]$.
\end{proof}
\begin{dfn}
If $X$ is a HK manifold of type $\Kum_n$ we let  $\Trans(X)\lhd  \Aut^0(X)$ be the unique abelian subgroup   of index $2$. If $X$ is of type  OG6, we let  $\Trans(X):=\Aut^0(X)$. The elements of   $\Trans(X)$ are the \emph{translations of $X$}.
\end{dfn}
\begin{expl}\label{expl:transkum}
 Let $A$ be a compact complex torus of dimension $2$. By Example~\ref{expl:eccaut} we have a natural identification 
 $\Trans(K_n(A))=A[n+1]$. 
\end{expl}
\subsection{Theta groups for HK manifolds of Kummer or OG6 type}\label{subsec:tetadifib}
\setcounter{equation}{0}
In the present subsection $X$ is a HK manifold of Kummer or OG6 type. Following Mumford~\cite{mum-eq-abvars} we define the theta group of 
 a simple sheaf  $\cF$ on $X$ under the assumption that
\begin{equation}\label{keyassumption}
g^{*}(\cF)\cong\cF\quad \forall g\in \Trans(X).
\end{equation}
\begin{dfn}
Let $\cF$ be a simple sheaf  on $X$ such that \eqref{keyassumption} holds. The \emph{theta group} $\cG(\cF)$ 
is  the set of couples $(g,\phi)$ where $g\in\Trans(X)$ and $\phi\colon \cF\overset{\sim}{\lra}g^{*}(\cF)$ is an isomorphism. The product is defined by 
\begin{equation}\label{prodinth}
(g_1,\phi_1)\cdot (g_2,\phi_2):=(g_1 g_2,g_2^{*}(\phi_1)\circ \phi_2).
\end{equation}
\end{dfn}
\begin{expl}
If $L$ is a line bundle, then $g^{*}(L)\cong L$ for all $g\in\Aut^0(X)$. Since $L$ is simple,  the theta group $\cG(L)$ is defined.    
\end{expl}
\begin{expl}
In~\cite{og:modonkum} we have constructed rank $4$ slope stable vector bundles $\cF$ on a generic polarized HK fourfold $(X,h)$ with $\divisore(h)=2$ and $q_X(h)\equiv -6\pmod{16}$ or  
$\divisore(h)=6$ and $q_X(h)\equiv -6\pmod{144}$ such that $c_1(\cF)=h$ and $\Delta(\cF)=c_2(X)$, where 
$\Delta(\cF)=8c_2(X)-3 c_1(X)^2$ is the discriminant of $\cF$. We have proved that for such vector bundles $g^{*}(\cF)\cong\cF$ for all $g\in\Aut^0(X)$  (op.cit.). Hence  the theta group $\cG(\cF)$ is defined.    
\end{expl}
\begin{expl}
The tangent 
bundle $\Theta_X$  is stable with respect to any K\"ahler metric, in particular it is simple. Since $g^{*}(\Theta_X)\cong\Theta_X$ for all $g\in \Aut^0(X)$, the theta group $\cG(\Theta_X)$ is defined. 
\end{expl}

The homomorphism $\cG(\cF)\to  \Trans(X)$ defined by $(g,\phi)\mapsto g$ 
gives an exact sequence of  groups
\begin{equation}\label{thetahk}
1\lra \CC^{*}\lra \cG(\cF)\lra  \Trans(X)\lra 1.
\end{equation}
The above exact sequence gives rise to the \emph{commutator pairing}
\begin{equation}\label{commpair}
\begin{matrix}
\Trans(X)\times  \Trans(X) & \overset{e^{\cF}}{\lra} & \CC^{*}. \\
 (\alpha,\beta) & \mapsto & \wt{\alpha}\cdot\wt{\beta}\cdot \wt{\alpha}^{-1}\cdot\wt{\beta}^{-1}
\end{matrix}
\end{equation}
where $\wt{\alpha},\wt{\beta}\in \cG(\cF)$ are lifts of 
 $\alpha,\beta\in  \Trans(X)$ respectively. (Note: it is here that we want $\Trans(X)$ to be abelian). For fixed $\beta\in  \Trans(X)$ the maps $\Trans(X)\to\CC^{*}$ defined by 
 $\alpha\mapsto e^{\cF}(\alpha,\beta)$ and 
 $\alpha\mapsto e^{\cF}(\beta,\alpha)$ are characters, and moreover 
 $e^{\cF}$ is skew symmetric, i.e.~$e^{\cF}(\alpha,\alpha)=1$. 
 The commutator pairing defines a homomorphism $E^{\cF}\colon \Trans(X)\to \wh{\Trans}(X)$, where $\wh{\Trans}(X)$ is the group of characters of $\Trans(X)$, by setting $E^{\cF}(\alpha)(\beta):=e^{\cF}(\alpha,\beta)$. The  commutator pairing is \emph{non degenerate} if $E^{\cF}$ is an isomorphism. Below we collect a few observations regarding the commutator pairing.
\begin{rmk}\label{rmk:solltriv}
Let $H< T(X)$ be a subgoup. The action of $H$ on $X$ lifts to an action on $\cF$ if and only if the restriction to $H$ of the commutator pairing is trivial, see p.~293 in~\cite{mum-eq-abvars}.
\end{rmk}
\begin{expl}\label{expl:tancom}
 The commutator pairing of $\cG(\Theta_X)$  is trivial because the action of $T(X)$ on $X$ lifts, via the differential, to an action on $\Theta_X$. 
\end{expl}
\begin{rmk}\label{rmk:autbir}
Let $f\colon Y\dra X$  be a birational (i.e.~bimeromorphic) map between HK manifolds. If $\varphi\in\Aut^0(X)$, then the induced birational map $f^{-1}\circ\varphi\circ f$ is regular, see  the proof of Thm 2.1 in~\cite{hass-tschink-lag-planes}.
Since $f^{-1}\circ\varphi\circ f$ acts trivially on $H^2(Y)$, we get
 a natural isomorphism 
\begin{equation}\label{amalia}
\begin{matrix}
\Aut^0(X) & \overset{\sim}{\lra} & \Aut^0(Y) \\
f & \mapsto & f^{-1}\circ\varphi\circ f
\end{matrix}
\end{equation}
\end{rmk}
\begin{rmk}\label{rmk:stessocomm}
Let $f\colon Y\dra X$  be a birational  map between HK manifolds of Kummer type or of type OG6. 
The isomorphism in~\eqref{amalia} restricts to an   isomorphism $T(X)\overset{\sim}{\lra} T(Y)$. Since $X,Y$ have trivial canonical line bundles, there exist  open subsets  $V\subset Y$ and $U\subset X$ with complements of codimension at least $2$ such that
$f$ is regular on $V$ and it defines an isomorphism $U\overset{\sim}{\lra} V$. 
Hence pull-back defines an isomorphism $f^{*}\colon\Pic(X)\overset{\sim}{\lra}\Pic(Y)$. If $L$ is a line bundle on $X$, then the isomorphism $T(X)\overset{\sim}{\lra} T(Y)$  
 lifts to an isomorphism of theta groups $\cG(L)\overset{\sim}{\lra} \cG(f^{*}(L))$ because $L_{|U}$ is identified  with $f^{*}(L)_{|V}$. In particular the commutator pairings of $\cG(L)$ and $\cG(f^{*}(L))$ are isomorphic.

\end{rmk}
\begin{rmk}\label{rmk:variacom}
Let $f\colon\cX\to T$ be a family of HK manifolds of Kummer  or  OG6 type over a connected base $T$, and assume that $\cL$ is a line bundle on $\cX$. If $a,b\in T$ then the commutator pairings $e^{L_a}$ on $T(X_a)$ 
and $e^{L_b}$  on $T(X_a)$ (here $X_p:=f^{-1}(p)$) are isomorphic. More precisely, any arc 
$\gamma\colon [0,1]\to T$ starting at $a$ and ending at $b$ determines an isomorphism 
$\gamma_{*}\colon T(X_a)\overset{\sim}{\lra} T(X_b)$ (see~\cite{hass-tschink-lag-planes}), and we have
$e^{L_a}(\alpha,\beta)=e^{L_b}(\gamma_{*}(\alpha),\gamma_{*}(\beta))$. In fact the commutator map varies continuously, and since it takes values in a finite subgroup of $\CC^{*}$ (because $T(X)$ is finite) it follows that it is locally constant.
\end{rmk}
\begin{rmk}
Let $L_1,L_2$ be line bundles on $X$, and let $L:=L_1 \otimes L_2$. For 
$(\alpha,\beta)\in T(X)\times T(X)$ we have
\begin{equation}\label{tenspair}
e^L(\alpha,\beta)=e^{L_1}(\alpha,\beta)\cdot e^{L_2}(\alpha,\beta). 
\end{equation}
\end{rmk}
\subsection{The Heisenberg group}\label{sec:heisengrp}
\setcounter{equation}{0}
If  the commutator pairing of $\cG(\cF)$ is non degenerate, then $\cG(\cF)$ is isomorphic to a Heisenberg group defined as follows. Let $1\le d_1|d_2|\ldots|d_g$ be natural numbers, let $\dd:=(d_1,\ldots,d_g)$, and let
\begin{equation}
J(\dd):=\ZZ/(d_1)\oplus\ldots\oplus\ZZ/(d_g).
\end{equation}
Let $\wh{J}(\dd)$ be the group of characters of $J(\dd)$.
\begin{dfn}
The \emph{Heisenberg group} $\cH(\dd)$ is the set $\CC^{*}\times J(\dd)\times\wh{J}(\dd)$ with the group operation defined by
\begin{equation}
(\alpha,x,f)\cdot(\beta,y,g):=(\alpha\cdot\beta\cdot g(x),x+y,f\cdot g).
\end{equation}
\end{dfn}
The forgetful map $\cH(\dd)\lra  J(\dd)\times\wh{J}(\dd)$ is a homomorphism of groups, and it fits into an exact sequence of groups
\begin{equation}\label{estcentr}
1\lra \CC^{*}\lra \cH(\dd)\lra   J(\dd)\times\wh{J}(\dd)\lra 1
\end{equation}
which gives rise to a commutator pairing
\begin{equation}
\left(J(\dd)\times\wh{J}(\dd)\right)\times \left(\wh{J}(\dd)\times J(\dd)\right)\overset{e^{\cH(\dd)}}{\lra} \CC^{*}.
\end{equation}
The next result follows from Corollary of Th.~1, p.~294 in~\cite{mum-eq-abvars}.
\begin{prp}\label{prp:senondeg}
Let $X$ be a HK manifold of type $Kum_n$ or of type $\OG$. Let $\cF$ be a simple sheaf on $X$ such that~\eqref{keyassumption} holds, and such that the commutator pairing $e^{\cF}$ is non degenerate. Then, if $X$ is 
of type $Kum_n$ there exists an isomorphism $\cG(\cF)\overset{\sim}{\lra} \cH(n+1,n+1)$, and if $X$ is  of type $\OG$ 
 there exists an isomorphism $\cG(\cF)\overset{\sim}{\lra} \cH(2,2,2,2)$. In both cases the isomorphism can be chosen so that the exact sequence in~\eqref{thetahk} is isomorphic to the exact sequence in~\eqref{estcentr}.
\end{prp}
Next we consider representations of $\cH(\dd)$. 
\begin{dfn}
The \emph{Schr\"odinger representation} of $\cH(\dd)$ is given by
\begin{equation}
\begin{matrix}
\cH(\dd)\times \CC^{J(\dd)} & \lra &  \CC^{J(\dd)}  \\
((\alpha,x,f),\varphi) & \mapsto & \left(y \mapsto \alpha\cdot f(y)\cdot\varphi(x+y)\right)
\end{matrix}
\end{equation}
\end{dfn}
The key result about representations is the following.
\begin{prp}[Prop.~3, p.~295 in~\cite{mum-eq-abvars}]\label{prp:tutteschrod}
Let  $\rho\colon \cH(\dd)\to \GL(V)$ be a finite dimensional representation   of $\cH(\dd)$ such that $\rho(\alpha,0,0)=\alpha\Id_V$ for every $\alpha\in\CC^{*}$. Then $\rho$ is a direct sum of copies of the Scr\"odinger representation. 
\end{prp}
\subsection{The theta group and  global sections}\label{sec:reponsec}
\setcounter{equation}{0}
Let $X$ be a HK manifold of type $Kum_n$ or of type OG6, and let $\cF$ be a simple sheaf on $X$ such that~\eqref{keyassumption} holds. The action of $T(X)$ on $\PP(H^0(X,\cF))$ lifts to an action of $\cG(\cF)$ on 
$H^0(X,\cF)$ as follows:
\begin{equation}\label{jazz}
\begin{matrix}
\cG(\cF)& \overset{\Psi}{\lra} & \GL(H^0(X,\cF)) \\
(g,\varphi) & \mapsto & \sigma\mapsto (g^{-1})^{*}(\varphi(\sigma).
\end{matrix}
\end{equation}
Note that if $\alpha\in\CC^{*}$ then the element $(\Id_X,\alpha)\in\cG(\cF)$ acts as $\alpha \Id_{H^0(X,\cF)}$. Now assume that the commutator pairing of $\cG(\cF)$ is non degenerate, and hence $\cG(\cF)$ is isomorphic to a  Heisenberg group by Proposition~\ref{prp:senondeg}.  Then the $\cG(\cF)$ representation $H^0(X,\cF)$ is isomorphic to a direct sum of Schr\"odinger representations by Proposition~\ref{prp:tutteschrod}
\begin{rmk}\label{rmk:unasez}
 If  $H^0(X,\cF)$ is non zero then 
 $e^{\cF}$ may be read off from the representation of $\cG(\cF)$ on  $H^0(X,\cF)$. More precisely, for $i\in\{1,2\}$ let 
 $g_i\in T(X)$, and let $\wt{g}_i=(g_i,\varphi_i)$ be a lift of $g_i$ to $\cG(\cF)$. Then
\begin{equation}\label{lalaland}
\Psi(\wt{g}_1)\circ \Psi(\wt{g}_2)\circ \Psi(\wt{g}_1)^{-1}\circ \Psi(\wt{g}_2)^{-1}=
(\Id_{\cF},e^{\cF}(g_1,g_2)\cdot\Id_{H^0(X,\cF)}).
\end{equation}
 For example, if 
 $h^0(X,\cF)=1$ it follows that  the commutator pairing is trivial, i.e.~$\cG(\cF)$ is the direct product $\CC^{*}\times T(X)$. More generally, suppose that $H<T(X)$ is a subgroup which acts trivially on  $\PP(H^0(X,\cF))$. If  $h\in H$ and $\wt{h}\in\cG(\cF)$ is a lift of $h$, then $\Psi(\wt{h})$ is a multiple of $\Id_{H^0(X,\cF)}$, and 
hence the equality in~\eqref{lalaland} shows that $e^{\cF}(h,g)=1$ for all $g\in\cG(\cF)$. 
Thus $H$ is in the kernel of  $e^{\cF}$.
\end{rmk}
\section{The commutator pairing for HK manifolds of Kummer type}\label{sec:compkum}
\setcounter{equation}{0}
\subsection{Preliminaries on generalized Kummers}\label{subsec:recaponkumm}
\setcounter{equation}{0}
Let $A$ be an abelian surface. The generalized Kummer $K_n(A)$ is the fiber over $0$ of the map $A^{[n+1]}\to A$ given by the composition
\begin{equation*}
A^{[n+1]}\overset{\mathfrak h}{\lra} A^{(n+1)}\overset{\sigma}{\lra} A
\end{equation*}
where  ${\mathfrak h}[Z]:=\sum_{a\in A}\ell(\cO_{Z,a})(a)$ is the Hilbert - Chow map and $\sigma$ is the summation map in the group $A$, i.e.~$\sigma((a_1)+\ldots+(a_{n+1})):=a_1+\ldots+a_{n+1}$. Here and in the rest of the paper we denote by $(a)$ the generator of the group of $0$ cycles on $A$ that corresponds to the point $a\in A$. Hence  if $k_1,\ldots,k_{n+1}$ are integers and $a_1,\ldots,a_{n+1}\in A$ then $k_1(a_1)+\ldots+k_{n+1}(a_{n+1})$ is a $0$ cycle while $k_1 a_1+\ldots+k_{n+1}a_{n+1}$ is an element of $A$. 

The cohomology group $H^2(K_n(A);\ZZ)$ is described as follows. There is a homomorphism $\mu_n\colon H^2(A)\to H^2(K_n(A))$ given by the composition
\begin{equation*}
H^2(A)\overset{s_{n+1}}{\lra} H^2(A^{(n+1)})\overset{{\gh}^{*}}{\lra} H^2(K_n(A))
\end{equation*}
where $s_{n+1}$ is the natural symmetrization map. The map $\mu_n$ is injective but not surjective because $\gh$ contracts the prime divisor (here we assume that $n\ge 2$)
\begin{equation*}
\Delta_n(A):=\{[Z]\in A^{[n+1]} \mid \text{$Z$ is not reduced}\}.
\end{equation*}
The cohomology class of $\Delta_n(A)$ is (uniquely) divisible by $2$ in integral cohomology. We let 
$\delta_n(A)\in H^2(A^{(n+1)};\ZZ)$ be the class such that
\begin{equation}
2\delta_n(A)=\cl(\Delta_n(A)).
\end{equation}
(Beware: $\delta_n(A)$ is \emph{not} the class of $\Delta_n(A)$.) One has 
\begin{equation}
\mu_n(H^2(A;\ZZ))\oplus_{\bot} \ZZ\delta_n(A),
\end{equation}
where orthogonality is with respect to the BBF quadratic form. Moreover the BBF quadratic form is given by 
\begin{equation}\label{kummerbbf}
q(\mu_n(\alpha)+x\delta_n(A))=(\alpha,\alpha)_A-2(n+1)x^2.
\end{equation}
(Here $(\alpha,\alpha)_A$ is the self intersection of $\alpha\in H^2(A)$.) 
We will use the following result of Mongardi-Pacienza. 
\begin{prp}[{\cite[Theorem~4.2]{mon-pac-uniruled}}]\label{prp:defpol}
Let $X$ be a HK manifold of type $Kum_n$ and let  $L$  be a line bundle on $X$. 
There exist a family  $f\colon \cX\to T$ of HK manifolds  over a connected base $T$, points $t_0,t_1\in T$ and a line bundle $\cL$ on $\cX$ with the following properties:
\begin{enumerate}
\item[(a)]
 $X_{t_0}:=f^{-1}(t_0)$ is isomorphic to $X$, and  
  $\cL_{|X_{t_0}}$  is isomorphic to $L$. 
\item[(b)]
 $X_{t_1}:=f^{-1}(t_1)$ is isomorphic to a generalized Kummer $K_n(A)$, and 
\begin{equation}
 c_1(\cL_{|X_{t_1}})=\mu_n(a_1 e_1+a_2 e_2)+x\delta_n(A),
\end{equation}
where $e_1,e_2\in H^2(A;\ZZ)$ is a standard basis of a hyperbolic sublattice of $H^2(A;\ZZ)$ (i.e.~$(e_i,e_i)_A=0$ and $(e_1,e_2)_A=1$), $a_1$ divides $a_2$, and of course the class  $a_1 e_1+a_2 e_2$ belongs to $\NS(A)$.
\end{enumerate}
\end{prp}
\subsection{Main result}
\setcounter{equation}{0}
Below is the main result of the present section.
\begin{thm}\label{thm:compairkum}
Let $X$ be a HK manifold of type $\Kum_n$ for $n\ge 2$, and let $L$ be a primitive line bundle on $X$, i.e.~such that  $c_1(L)$ is primitive. Let 
$f\colon \cX\to T$ be  a family of HK manifolds as in Proposition~\ref{prp:defpol}, and let $b_i:=\gcd\{a_i,n+1\}$ where $a_1,a_2$ are as in Item~(b) of 
Proposition~\ref{prp:defpol}. 
Then 
\begin{equation}\label{conucleo}
\wh{T}(X)/\im(E^L)\cong (\ZZ/(b_1))^2\oplus (\ZZ/(b_2))^2
\end{equation}
\end{thm}
\begin{rmk}
By Theorem~\ref{thm:compairkum} $\cG(L)$ is a Heisenberg group if and only if 
$b_1=b_2=1$. Since
\begin{equation}
\divisore(L)=\gcd\{a_1,2(n+1)\},
\end{equation}
$b_1=1$ if and only if $\divisore(L)=1$ or $\divisore(L)=2$ and $n$ is even. Since $q(L)/2=a_1 a_2-x^2(n+1)$, we get that if $b_1=1$ then $b_2=1$ if and only if $\gcd\{n+1,q(L)/2\}=1$. This shows that Theorem~\ref{thm:heisenkum} follows from Theorem~\ref{thm:compairkum}.
\end{rmk}
\begin{rmk}
For a prime $p$ and $a\in\QQ$, let $\ord_p(a)$ be the integer $m$ such that $a=p^m(b/c)$ where $b$ and $c$ are coprime to $p$ (we let 
$\ord_p(0):=+\infty$).
If $X$ is a HK manifold of type $\Kum_n$ and  $L$ is a line bundle on $X$, then
\begin{equation*}
b_1=
\begin{cases}
\divisore(L) & \text{if $\ord_2(\divisore(L))\le \ord_2(n+1)$}, \\
\frac{\divisore(L)}{2} & \text{if $\ord_2(\divisore(L))<\ord_2(n+1)$.}
\end{cases}
\end{equation*}
Let $p$ be a prime.  If $\ord_p(n+1)=0$ then $\ord_p(b_2)=0$. Suppose that  $\ord_p(n+1)>0$; if 
$\ord_p(\divisore(L))< \ord_p(n+1)$ (for $p=2$ equality is also allowed)
 then
\begin{equation}
\ord_p(b_2)=\ord_p\left(\frac{q_X(L)}{2\divisore(L)}\right).
\end{equation}
\end{rmk}
\begin{rmk}
Let $X$ be a HK manifold of type $\Kum_n$ for $n\ge 2$, and let $L$ be an ample primitive line bundle on $X$. Let $q_X(L)=2e$. By Kodaira vanishing and Britzke's formula for Huybrechts' HRR formula for HK manifolds of Kummer type, we have
\begin{equation}\label{hiriro}
h^0(X,L)=(n+1)\cdot{e+n\choose n}.
\end{equation}
Now suppose that the commutator pairing of $L$ is non degenerate, i.e.~that the hypotheses of Theorem~\ref{thm:heisenkum} hold. Then $H^0(X,L)$ is isomorphic to a direct sum of copies of the Heisenberg representation $\cH(n+1,n+1)$. Since $\cH(n+1,n+1)$ has dimension $(n+1)^2$, it follows that $h^0(X,L)$ must be a multiple of $(n+1)^2$, and by the equality in~\eqref{hiriro} this means that $(n+1)$ divides ${e+n\choose n}$. An elementary argument confirms that this is the case. Moreover we get that  $H^0(X,L)$  is the  Heisenberg representation if and only if $e=1$, i.e.~$q_X(L)=2$.
\end{rmk}

\subsection{The commutator pairing for generalized Kummers} 
\setcounter{equation}{0}
Let $A$ be a compact complex torus, and let $\ell\in \NS(A)$.  The commutator pairing $e^{(n+1)\ell}$ of a line bundle $L_A$ on $A$ such that $c_1(L_A)=(n+1)\ell$ is defined on $K((n+1)\ell)\times K((n+1)\ell)$, where 
$K((n+1)\ell)$ is the subgroup of translations $g$ of $A$ such that $g^{*}(L_A)\cong L_A$. Since $A[n+1]$ is a subgroup of  $K((n+1)\ell)$, it makes  sense to restrict $e^{(n+1)\ell}$ to $A[n+1]\times A[n+1]$. Recall that we have a natural identification $T(K_n(A))=A[n+1]$, see Example~\ref{expl:eccaut}.
\begin{prp}\label{prp:compatibili}
Keep notation as above, and let
 $L$ be  the line bundle on $K_n(A)$ such that 
$c_1(L)=\mu_n(\ell)$ (notation as in Subsection~\ref{subsec:recaponkumm}). Then $e^L$ is equal to the restriction of $e^{(n+1)\ell}$ to $A[n+1]\times A[n+1]$. 
\end{prp}
\begin{proof}
There exist a family of compact complex tori $f\colon \cA\to T$ over a connected base $T$, points $t_0,t_1\in T$ and a line bundle $\xi$ on $\cA$ with the following properties:
\begin{enumerate}
\item[(a)]
 The fiber $A_{t_0}:=f^{-1}(t_0)$ is isomorphic to $A$, while the fiber $A_{t_1}:=f^{-1}(t_1)$ is isomorphic to the product  $C_1\times C_2$ of elliptic curves.
\item[(b)]
The line bundle  $L_{t_0}$ on $K_n(A_{t_0})$ such that 
$c_1( L_{t_0})=\mu_n(c_1(\xi_{t_0}))$ is identified with $L$ via the isomorphism 
$K_n(A_{t_0})\overset{\sim}{\lra} K_n(A)$ induced by the isomorphism $A_{t_0}\overset{\sim}{\lra} A$ of Item~(a).
\item[(c)]
The line bundle $\xi_{t_1}$  is of product type,
 i.e.~$\xi_{t_1}\cong \cO_{C_1}(D_1)\boxtimes \cO_{C_2}(D_2)$ for divisors $D_1,D_2$ on 
 $C_1,C_2$ respectively.
\end{enumerate}
Considering the relative family of generalized Kummers over $T$ with fiber $K_n(A_t)$ over $t$ and recalling 
Remark~\ref{rmk:variacom}, it follows that it suffices to prove the proposition under the additional assumption that $A=C_1\times C_2$ and $\ell$ is of product type, 
i.e.~$\ell=\rho_1^{*}(a_1 \eta_1)+\rho_2^{*}(a_2 \eta_2)$ where $\rho_i\colon C_1\times C_2\to C_i$ is  projection, and
$\eta_i\in H^2(C_i;\ZZ)$ is the fundamental class. 
Hence by~\eqref{tenspair} we may assume that $\ell=\rho_1^{*}(\eta_1)$. We are reduced to proving the Proposition in this particular case.

Let $A^{(n+1),0}\subset A^{(n+1)}$ and  $C_1^{(n+1),0}\subset C_1^{(n+1)}$ be the subgroups of cycles summing up to $0$ in $A$ and $C_1$ respectively.  Let $p\in C_1$ be the zero of the addition law. Then $C_1^{(n+1),0}$ is naturally identified with
 $|\cO_{C_1}((n+1)p)|\cong\PP^n$. 
The composition of the Hilbert-Chow map $K_n(A)\to A^{(n+1),0}$ and the projection $A^{(n+1),0}\to C_1^{(n+1),0}$ is a Lagrangian fibration $\pi\colon K_n(A)\to C_1^{(n+1),0}$. We have $L\cong \pi^{*}(\cO_{\PP^n}(1))$. In particular the action of $C_2[n+1]$ on $|L|$ is trivial and hence
  $C_2[n+1]$ is contained in the kernel of $E_L$, see Remark~\ref{rmk:unasez}. 
Hence $e^L$ defines a skew-symmetric pairing on $C_1[n+1]$ with values in $\CC^{*}$. We claim that this pairing is the same as the commutator pairing of the line bundle 
$\cO_{C_1}((n+1)p)$ on $K(\cO_{C_1}((n+1)p))=C_1[n+1]$. In fact the latter can be computed by the action of the theta group 
$\cG( \cO_{C_1}((n+1)p))$ on the space of sections $H^0(C_1,\cO_{C_1}((n+1)p))$, and since the pull-back by $\pi$ defines an isomorphism 
\begin{equation*}
H^0(C_1,\cO_{C_1}((n+1)p))=H^0(\PP^n,\cO_{\PP^n}(1))\overset{\pi^{*}}{\lra} H^0(K_n(A),L),
\end{equation*}
our claim follows. Since the first Chern class of $\cO_{C_1}((n+1)p)\boxtimes \cO_{C_2}$ is equal to $n\ell$ we are done.
\end{proof}
\begin{crl}\label{crl:accoppiamento}
Keep notation and hypotheses as in Proposition~\ref{prp:compatibili}, and let $a_1 ,a_2$ be the elementary divisors of 
$\ell$, where $a_1 | a_2$. 
 Then 
\begin{equation}\label{cirocirillo}
\wh{T}(K_n(A))/\im(E^L)\cong (\ZZ/(b_1))^2\oplus (\ZZ/(b_2))^2
\end{equation}
 where $b_i:=\gcd\{n+1,a_i\}$.
\end{crl}
\begin{proof}
By Proposition~\ref{prp:compatibili} the commutator pairing $e^L$ is equal to the restriction of $e^{(n+1)\ell}$ to $A[n+1]$. 
We have an orthogonal direct sum decomposition
\begin{equation*}
K((n+1)\ell)\cong (\ZZ/(a_1(n+1)))^2\oplus_{\bot} (\ZZ/(a_2(n+1)))^2,
\end{equation*}
and an adapted basis $\{\wt{\alpha_1},\wt{\beta_1},\wt{\alpha_2},\wt{\beta_2}\}$ such that 
$e^{(n+1)\ell}(\wt{\alpha_i},\wt{\beta_i})=\epsilon_{a_i(n+1)}$ is a primitive $a_i(n+1)$-th root of $1$. We may choose the basis such that  $\epsilon^{a_1}_{a_1(n+1)}=\epsilon^{a_2}_{a_2(n+1)}$, call it $\epsilon$. Of course 
$\epsilon$  is a primitive $(n+1)$-th root of $1$. Let $\alpha_i:=a_1\wt{\alpha}_i$ and  $\beta_i:=a_i\wt{\beta}_i$.  
Then $\{\alpha_1,\beta_1,\alpha_2,\beta_2\}$ is a basis of $A[n+1]$, and 
$e^{(n+1)\ell}(\wt{\alpha_i},\wt{\beta_i})=\epsilon^{a_i}$.
The isomorphism in~\eqref{cirocirillo} follows at once.
\end{proof}
\begin{prp}\label{prp:nonconta}
Let $A$ be a compact complex torus and let $L$ a line bundle on $K_n(A)$ such that $c_1(L)$ is a multiple of 
$\delta_n$. The commmutator pairing of $L$ is trivial.
\end{prp}
\begin{proof}
The linear system 
$|\Delta_n|$  consists of the single divisor $\Delta_n$ and hence
 the commmutator pairing of $\cO_{K_n(A)}(\Delta_n)$   is trivial by Remark~\ref{rmk:unasez}. This proves that $e^L$ is trivial if $c_1(L)=x\delta_n$ with $x$ even but not for odd $x$ (well, it does if $n$ is even because in that case there are no non trivial $2$ torsion elements of $A[n+1]$). Now suppose that $c_1(L)=\delta_n$. Let $\wt{K}_n(A)\to K_n(A)$ be the double cover ramified over $\Delta_n$. Then 
\begin{equation*}
\rho_{*}(\cO_{\wt{K}_n(A)})=\cO_{K_n(A)}\oplus L^{-1} 
\end{equation*}
and $L^{-1}$ is the $(-1)$ eigensheaf for the natural  action of the covering involution.
It follows that in order to prove 
  that $e^L$ is trivial it suffices to lift the action of $T(K_n(A))$ on $K_n(A)$ to an action on $\wt{K}_n(A)$ (recall Remark~\ref{rmk:solltriv}). This is done by considering the isospectral Hilbert scheme $X_{n+1}(A)$ obtained by blowing up the big diagonal in $A^{n+1}$ and the finite map $X_{n+1}(A)\to A^{[n+1]}$, see~\cite{haiman}. Let $Y_n(A)\subset X_{n+1}(A)$ be the inverse image of $K_n(A)\subset A^{[n+1]}$.
   The action of the permutation group 
  $\cS_{n+1}$ on $A^{n+1}$ lifts to an action on  $X_{n+1}(A)$ and also on $Y_n(A)$. The double cover
   $\wt{K}_n(A)\to K_n(A)$ is identified with the double cover
\begin{equation*}
Y_n(A)/ \cA_{n+1} \lra K_n(A)
\end{equation*}
where  $\cA_{n+1}\lhd\cS_{n+1}$ is the alternating group. The group of translations of $A$ acts on $A^{n+1}$ and it maps the big diagonal to itself, hence the action lifts to an action on $X_{n+1}(A)$. Note that the action commutes with the permutation action.
The subgroup of translations in $A[n+1]$  acts on $Y_n(A)$ and hence also on $Y_n(A)/ \cA_{n+1}$.  (Note that we do not need to go through the highly non trivial results of Haiman. In fact it suffices to work away from the codimension $2$ subset of $K_n(A)$ parametrizing subschemes $Z$  such that  $\gh(Z)$   has a point of multiplicity greater than $2$, where the statement about $X_{n+1}(A)$ being the blow up of the big diagonal is elementary.)
\end{proof}
\subsection{Proof of Theorem~\ref{thm:compairkum}}
\setcounter{equation}{0}
By  Remark~\ref{rmk:variacom} and Proposition~\ref{prp:defpol} we may assume that $X=K_n(A)$ and 
$c_1(L)=\mu_n(\ell)+x\delta_n$, where $\ell\in\NS(A)$ and $x\in\ZZ$. In this case the statement of 
Theorem~\ref{thm:compairkum} holds by Corollary~\ref{crl:accoppiamento} and Proposition~\ref{prp:nonconta}
\qed

\section{The commutator pairing for HK manifolds of type OG6}\label{sec:contocomm}
\setcounter{equation}{0}
\subsection{Main result}
\setcounter{equation}{0}
Before stating the  main result, we note that if $X$ is a HK manifold of type $\OG$ then   a primitive element of 
$H^2(X;\ZZ)$ has divisibility $1$ or $2$ (see~\cite{rapagnetta} or~\eqref{reticolog6}).
\begin{thm}\label{thm:compairog6}
Let $X$ be a HK manifold of type $\OG$, and let $L$ be a primitive line bundle on $X$.   
\begin{enumerate}
\item
If $\divisore(L)=1$ and $q_X(L)$ is not divisible by $4$ then  $e^L$ is non degenerate.
\item
If $\divisore(L)=1$ and $q_X(L)$ is divisible by $4$ then  $\wh{T}(X)/\im(E_L)\cong (\ZZ/(2))^4$.
\item
If $\divisore(L)=2$  then $e^L$ is trivial.
\end{enumerate}
\end{thm}
The proof for the case $\divisore(L)=2$ is in Subsection~\ref{subsec:theend},  the proof for the case
$\divisore(L)=1$ is in Subsection~\ref{sec:divuno}. 

\begin{rmk}
Theorem~\ref{thm:heisenog6} follows at once from Theorem~\ref{thm:compairog6}.
\end{rmk}
\begin{rmk}
Let $X$ be a HK manifold of type  $\OG$, and let $L$ be an ample primitive line bundle on $X$. Let $q_X(L)=2e$. By Kodaira vanishing and  Huybrechts' HRR formula for HK manifolds of  type  $\OG$, we have
\begin{equation}\label{akihiko}
h^0(X,L)=4\cdot{e+3\choose 3}.
\end{equation}
Now suppose that the commutator pairing of $L$ is non degenerate, i.e.~that the hypotheses of Theorem~\ref{thm:heisenog6} hold. Then $H^0(X,L)$ is isomorphic to a direct sum of copies of the Heisenberg representation $\cH(2,2,2,2)$. Since $\cH(2,2,2,2)$ has dimension $16$, it follows that $h^0(X,L)$ must be a multiple of $16$. An elementary argument confirms that this is the case. Moreover we get that  $H^0(X,L)$  is the  Heisenberg representation if and only if $e=1$, i.e.~$q_X(L)=2$.
\end{rmk}

\subsection{Preliminaries on HK manifolds of  type OG6}\label{subsec:varog6}
\setcounter{equation}{0}
HK manifolds of type OG6 are $6$-dimensional and they belong to a single deformation class. The first examples where constructed by the author in~\cite{og6} as  symplectic desingularizations of an Albanese fiber of a suitably chosen singular moduli spaces of semistable sheaves of rank $2$ on Jacobians of genus $2$ curves. M.~Lehn and C.~Sorger~\cite{lehn-sorg-singog} examined in detail the singularities of the relevant moduli spaces and
  proved that one can construct a similar symplectic desingularization of an Albanese fiber of moduli spaces of semistable sheaves on an abelian surface if the Mukai vector is of the form $v=2w$ where $w^2=2$ (the square is with respect to the Mukai pairing). 
Rapagnetta and Perego~\cite{per-rap-ogms}  proved  that all HK varieties obtained this way 
are deformation equivalent, and Rapagnetta~\cite{rapagnetta}  determined their BBF quadratic form.   Following is a more detailed exposition. Let $A$ be an abelian surface, and let  
\begin{equation}
\wt{H}(A;\ZZ):=H^0(A;\ZZ)\oplus H^2(A;\ZZ)\oplus H^4(A;\ZZ)
\end{equation}
be its Mukai lattice, where the Mukai quadratic form is defined by $(r,\alpha,s)^2:=\alpha^2-2rs$. Let 
\begin{equation}
v_0=(r_0,\alpha_0,s_0)\in 
\wt{\NS}(A)=H^0(A;\ZZ)\oplus \NS(A)\oplus H^4(A;\ZZ)
\end{equation}
be a Mukai vector  such that  $r_0\ge 0$ and moreover $\alpha$ is effective (non zero) if $r_0=0$. We assume that $v_0^2=2$. Let $v:=2v_0$ and let $h$ be a $v$-generic polarization of $A$ (if the Picard number of $A$ is $1$, then any polarization of $A$ is  $v$-generic). The moduli space $M_{v}(A,h)$ of $h$ Gieseker-Maruyama semistable sheafs $\cF$ on $A$ with 
$\ch(\cF)=v$ is irreducible of dimension $10$. There is an embedding of the symmetric square of $M_{v_0}(A,h)$ into  
$M_{v}(A,h)$
\begin{equation*}
\begin{matrix}
M_{v_0}(A,h)^{(2)} & \hra & M_{v}(A,h) \\
[\cF_1]+[\cF_2] & \mapsto & [\cF_1\oplus\cF_2]
\end{matrix}
\end{equation*}
whose image (of dimension $8$)  is the singular locus of  $M_{v}(A,h)$. 
Let  $\wt{M}_{v}(A,h)\to M_{v}(A,h)$ be the blow up of the singular locus $\Sigma_v(A,h)$. Then $\wt{M}_{v}(A,h)$ is smooth 
with Albanese variety isomorphic to $A\times A^{\vee}$. Let  $\wt{K}_{v}(A,h)$ be a fiber of 
(a chosen)  Albanese map $\wt{M}_{v}(A,h)\to A\times A^{\vee}$; then  $\wt{K}_{v}(A,h)$ is a HK variety of type OG6. 

Next we describe $H^2(\wt{K}_{v}(A,h);\ZZ)$ and the BBF quadratic form. Let $K_v(A,h)\subset M_v(A,h)$ be the image of
 $\wt{K}_{v}(A,h)$ under the blow up (deingularization) map.  First there is the Donaldson-Mukai-Le Potier homomorphism
\begin{equation}
\theta_v\colon v^{\bot}\lra H^2(K_v(A,h);\ZZ),
\end{equation}
where $ v^{\bot}\subset \wt{H}(A;\ZZ)$ is the orthogonal of $v$ (i.e.~of $v_0$) with respect to the Mukai quadratic form. 
There are some choices to be made in defining $\theta_v$; we choose to follow the definition in \cite{Yos:moduli}, see~(1.6) op.~cit.
The pull back via the blow up map $\wt{K}_{v}(A,h)\to K_{v}(A,h)$ defines an  isometry  
\begin{equation}
\wt{\theta}_v\colon v^{\bot}\lra H^2(\wt{K}_v(A,h);\ZZ),
\end{equation}
onto a saturated sublattice of $H^2(\wt{K}_v(A,h);\ZZ)$. The orthogonal  complement is generated by a class $\alpha_v$ of BBF square $(-2)$ 
such that $2\alpha_v$ is the  Poincar\`e dual of the exceptional divisor $\wt{\Sigma}_v(A,h)$ of the blow up map 
$\wt{K}_{v}(A,h)\to K_{v}(A,h)$. Moreover  $H^2(\wt{K}_{v}(A,h);\ZZ)$ is generated over $\ZZ$ by  $\im(\wt{\theta}_v)$ and $\alpha_v$.
The upshot is that we have an isometry of lattices (see Theorem 3.1 in~\cite{per-rap-fact})
\begin{equation}\label{reticolog6}
\begin{matrix}
v^{\bot}\oplus_{\bot}(-2) & \overset{\sim}{\lra} & H^2(\wt{K}_{v}(A,h);\ZZ) \\
(x,t) & \mapsto & \wt{\theta}_v(x)+t\alpha_v
\end{matrix}
\end{equation}
In particular, if $X$ is HK manifold of type OG6, then 
$H^2(X;\ZZ)$ equipped with the BBF quadratic form is isometric to the lattice
\begin{equation}\label{retog6}
\Lambda_{\OG}:= U_1\oplus_{\bot} U_2\oplus_{\bot} U_3\oplus_{\bot} \la g_1\ra\oplus_{\bot} \la g_2\ra.
\end{equation}
where each $U_i$ is a hyperbolic plane, and $(g_j,g_j)=-2$ for $j\in\{1,2\}$. 
We record here the following fact.
\begin{prp}\label{prp:vettog6}
If $a\in \Lambda_{\OG}$ is primitive 
(in particular non zero), then one of the following hold:
\begin{enumerate}
\item[(I)] 
$\divisore(a)=1$,
\item[(II)]
$\divisore(a)=2$ and $(a,a)\equiv -2\pmod{8}$,
\item[(III)] 
$\divisore(a)=2$ and $(a,a)\equiv -4\pmod{8}$.
\end{enumerate}
\end{prp}
\subsection{Varieties of type OG6 corresponding to $v=(0,2h,-2)$}\label{subsec:og6rap}
\setcounter{equation}{0}
Let $J$ be an abelian surface with a  principal polarization $h$. We make the following assumption:
\begin{equation}\label{picarduno}
\NS(J)=\ZZ h.
\end{equation}
In particular $J$ is the Jacobian of a (smooth projective) curve $C$ of genus $2$.

Let  $v=(0,2h,-2)$. We denote the moduli space $M_v(J,h)$ by $M_v(J)$ (there is a unique polarization).
Points of $M_v(J)$  parametrize $h$-semistable sheaves $\iota_{C,*}(\xi)$ where $\iota_{C}\colon C\hra J$ is the inclusion of a curve with Poincar\'e dual $2h$ and $\xi$ is a pure sheaf on $C$ of degree $2$. 

 Let $\Pic^{2h}(J)$ be the component of the Picard scheme of $J$ parametrizing line bundles $\cL$ with $c^B_1(\cL)=2h$, where $c^B_1(\cL)\in H^{2}(J;\ZZ)$ is the first Chern class in Betti cohomology. Let $\sigma\colon CH_2(J)\to J$ be the homomorphism defined (at the level of cycles)  by 
$\sigma\left(\sum_i m_i (x_i)\right):=\sum_i m_i x_i$, where the first sum is a formal sum, while the second one is the sum in the group $J$.
Consider the map
\begin{equation}\label{franbei}
\begin{matrix}
M_v(J) & \overset{\alb}{\lra} & \Pic^{2h}(J)\times J \\
[\cF] & \mapsto & (c_1(\cF), \sigma(c_2(\cF))
\end{matrix}
\end{equation}
where Chern classes are taken in the Chow ring of $J$ and we identify $\CH^1(J)$ with $\Pic(J)$. An Albanese fibration of 
$\wt{M}_v(J)$,  denote it by $\Alb$, is provided by the composition 
$\wt{M}_v(J)\lra M_v(J)\overset{\alb}{\lra} \Pic^{2h}(J)\times J $.
Now let 
$\Theta_J$ be a symmetric principal polarization of $J$ with cohomology class $h$, and let
\begin{equation}\label{eccokappa}
K_v(J) :=\alb^{-1}([2\Theta_J],0),\quad \wt{K}_v(J) :=\Alb^{-1}([2\Theta_J],0).
\end{equation}
The HK variety $\wt{K}_v(J)$ is examined in detail in the papers~\cite{rapagnetta-og6} and~\cite{monrapsac-og6}. 
To be precise in~\cite{monrapsac-og6} the Mukai vector is $(0,2h,2)$, but tensorization by $\cO_J(\Theta_J)$ defines an isomorphism between the moduli spaces $M_{(0,2h,-2)}(J)$ and $M_{(0,2h,2)}(J)$. 

Here we collect results  that will be needed later. 
Let $\beta_v,\gamma_v\in \NS(\wt{K}_v(J))=\NS(\wt{K}_v(J))$ be defined by
\begin{equation}
\beta_v:=\wt{\theta}_v(2,-h,1),\qquad \gamma_v:=\wt{\theta}_v(2,-h,0).
\end{equation}
By the isometry in~\eqref{reticolog6} and the equality in~\eqref{picarduno}, we have an orthogonal direct sum decomposition
\begin{equation}\label{abigamma}
\NS(\wt{K}_v(J))=\ZZ\alpha_v\oplus \ZZ\beta_v \oplus  \ZZ\gamma_v.
\end{equation}
Moreover we have
\begin{equation}\label{quadrati}
(\alpha_v,\alpha_v)=-2,\quad  (\beta_v,\beta_v)=-2,\quad (\gamma_v,\gamma_v)=2.
\end{equation}
The geometric meaning of class the $\alpha_v$ (or rather $2\alpha_v$) is clear by definition.  
We give geometric realizations of the classes $\gamma_v-\beta_v$ and $\beta_v-\alpha_v$. 
Since sheaves parametrized by $K_v(J)$ have rank $0$, we have 
 a  map 
\begin{equation}
\begin{matrix}
K_v(J) & \overset{\pi_v}{\lra} & |2\Theta_J|\cong\PP^3\\
[\cF] & \mapsto & \Supp_{\det}(\cF)
\end{matrix}
\end{equation}
where $\Supp_{\det}(\cF)$ (the determinantal support of $\cF$) is the curve defined by the $0$-th Fitting ideal of $\cF$ (and is well-defined also for properly semistable sheaves). 
Let 
\begin{equation}
\wt{\pi}_v\colon \wt{K}_v(J) \to |2\Theta_J|\cong\PP^3
\end{equation}
 be the  composition of  the desingularization map and $\pi_v$. Then $\wt{\pi}_v$ is a Lagrangian fibration. By~\cite{LePotier:Fitting} (see Section 2.3) we have that 
\begin{equation}
\gamma_v-\beta_v=\wt{\theta}_v(0,0,-1)=c_1\left(\pi_v^*\cO_{\PP^3}(1)\right).
\end{equation}
The geometric meaning of  the class $\beta_v-\alpha_v$ is contained in the proof of the result below. 
\begin{prp}\label{prp:ellevu}
The  line bundle  $\cL_v$  such that   $c_1(\cL_v)=\beta_v-\alpha_v$ has a unique global section up to rescaling.
\end{prp}
\begin{proof}
Let $w=(2,0,-2)$. Let $\wt{K}_w(2,0,-2)$ be the Albanese fiber of $\wt{\cM}_w(2,0,-2)$ lying over the point $(0,[\cO_J])\in \times \wh{J}$. 
In~\cite{rapagnetta-og6} one finds the definition of a birational map
\begin{equation}
\tau\colon \wt{K}_w(2,0,-2)\dra \wt{K}_v(0,2h,-2)
\end{equation}
such that
\begin{equation}\label{delstab}
\tau^{*}(\alpha_v)=\alpha_w.
\end{equation}
Let $\wt{B}\subset \wt{K}_w(2,0,-2)$ be the prime divisor whose generic point parametrizes a stable non locally free sheaf  on $J$. By Theorem~3.5.1 in op.cit.~we have 
\begin{equation}
(\cl(\wt{B}),\cl(\wt{B}))=-4,\qquad (\cl(\wt{B}),\alpha_w)=2.
\end{equation}
Hence by~\eqref{delstab} we get that
\begin{equation}
(\cl(\tau_{*}(\wt{B})),\cl(\tau_{*}(\wt{B})))=-4,\qquad (\cl(\tau_{*}(\wt{B})),\alpha_v)=2.
\end{equation}
It follows that 
\begin{equation}
(\cl(\tau_{*}(\wt{B}))+\alpha_v,\cl(\tau_{*}(\wt{B}))+\alpha_v)=-2,\qquad (\cl(\tau_{*}(\wt{B}))+\alpha_v,\alpha_v)=0.
\end{equation}
The orthogonal decomposition in~\eqref{abigamma}, the \lq\lq squares\rq\rq\ in~\eqref{quadrati}  and a straightforward computation show that $\pm \beta_v$ are the only classes in $\NS(\wt{K}_v(J))$ orthogonal to $\alpha_v$ and of square $-2$. 
Hence $\cl(\tau_{*}(\wt{B}))+\alpha_v=\pm \beta_v$, i.e.~$\cl(\tau_{*}(\wt{B}))=\pm \beta_v-\alpha_v$. Since 
$\tau_{*}(\wt{B})$ is an effective divisor, and $\gamma_v-\beta_v$ is the class of a movable divisor, we have
\begin{equation}
(\cl(\tau_{*}(\wt{B})), \gamma_v-\beta_v)\ge 0,
\end{equation}
and hence $\cl(\tau_{*}(\wt{B}))=\beta_v-\alpha_v$. In order to finish the proof it suffices to show that the divisor
$\wt{B}$ does not move. In fact there is an open dense subset of $\wt{B}$ which is a (smooth) 
conic fibration over a $4$ dimensional locally closed subset of $J\times M_{u}(J)$, where $u=(2,0,-1)$ (the \lq\lq conics\rq\rq\ are the generic fibers of the Uhlenbeck map $M_{w}(J)\to M_{w}(J)^{DUY}$). By adjunction  the restriction  to  a conic fiber $C$ of the normal bundle of $\wt{B}$ is 
the canonical line bundle $\omega_C$, and hence 
$\wt{B}$ does not move.
\end{proof}
\subsection{Translations of $\wt{K}_v(J)$ for $v=(0,2h,-2)$}\label{sec:transrap}
\setcounter{equation}{0}
One defines an action of $J[2]\times\wh{J}[2]$ on $\wt{K}_v(J)$ as follows. 
Let $[\cF]\in K_v(J)$. Let $x\in J[2]$ and let $\tau_x\colon J\to J$ be the translation by $x$. Then $\tau_x(\cF)$ is parametrized by 
a point of $K_v(J)$. In fact $\tau_x(\cF)$ is clearly $h$-semistable, moreover $c_1(\tau_x(\cF))=2\Theta_J$ 
(here and in what follows Chern classes are taken in the Chow ring) because $J[2]$ is the group of translations sending the rational equivalence class of $2\Theta_J$ to itself, and lastly $\sigma(c_2(\tau_x(\cF)))=0$ 
because $\deg c_2(\cF)=2$ (all we need is that $\deg c_2(\cF)$ is  even). Similarly, let $y\in\wh{J}[2]$, and let 
 $\cL_y$ be the corresponding line bundle on $J$;
  then $[\cF\otimes\cL_y]\in K_v(J)$. In fact $\cF\otimes\cL_y$ is clearly $h$-semistable, moreover $c_1(\cF\otimes\cL_y)=c_1(\cF)=2\Theta_J$, and lastly  
\begin{equation*}
\sigma(c_2(\cF\otimes\cL_y))=\sigma(c_2(\cF))+\sigma(c_1(\cF)\cdot c_1(\cL_y))=\sigma(2\Theta_J\cdot c_1(\cL_y))=0.
\end{equation*}
Thus we have an embedding 
\begin{equation*}
\begin{matrix}
J[2]\times\wh{J}[2] & \hra & \Aut(K_v(J)) \\
(x,y) & \mapsto & ([\cF]\mapsto [\tau_x(\cF)\otimes\cL_y])
\end{matrix}
\end{equation*}
Since the action maps the singular locus to itself, and the map $\pi\colon \wt{K}_v(J)\to K_v(J)$ is the blow up of the singular locus, we get an embedding $J[2]\times\wh{J}[2]  \hra  \Aut(\wt{K}_v(J))$. If $x\in J[2]$ 
we let $\lambda_x\colon \wt{K}_v(J)\to \wt{K}_v(J)$ be the  automorphism corresponding to $\tau_x$, and if  
$y\in\wh{J}[2]$ we let $\mu_y\colon \wt{K}_v(J) \to \wt{K}_v(J)$ be the automorphism corresponding to tensorization 
with $\cL_y$. Since one may define similarly  an action of $J\times\wh{J}$ on $\wt{M}_v(J)$, and each such automorphism of $\wt{M}_v(J)$ acts trivially on cohomology (it is homotopic to the identity), and since the homomorphism  $H^2(\wt{M}_v(J))\to H^2(\wt{K}_v(J))$ is surjective, we 
get an embedding 
\begin{equation}\label{taipei}
\begin{matrix}
J[2]\times\wh{J}[2] & \hra & \Aut^0(\wt{K}_v(J)) \\
(x,y) & \mapsto & \lambda_x\circ\mu_y 
\end{matrix}
\end{equation}
By Theorem.~5.2 in~\cite{mon-wand}, the map in~\eqref{taipei} is an isomorphism.
\subsection{The double cover  of $\wt{K}_v(J)$ for $v=(0,2h,-2)$}\label{sec:doppiohilb}
\setcounter{equation}{0}
Since $2\alpha_v=\cl(\wt{\Sigma}_v(J))$, there exists a double cover 
\begin{equation}\label{doppelganger}
\wt{Y}_v(J)\overset{\wt{\epsilon}_v}{\lra} \wt{K}_v(J)
\end{equation}
ramified over $\wt{\Sigma}_v(J)$, and such that 
\begin{equation}\label{autodecomp}
\wt{\epsilon}_{v,*}\left(\cO_{\wt{Y}_v(J)}\right)=\cO_{\wt{K}_v(J)}\oplus\cF_v,
\end{equation}
where $\cF_v$ is a line bundle, $c_1(\cF_v)=-\alpha_v$, and the addends  on the right hand side of~\eqref{autodecomp} are
the $\pm 1$ eigenspaces for the natural action  on the left hand side of the covering involution of $\wt{\epsilon}_v$.  
\begin{prp}\label{prp:sollevo}
The action of $\Aut^0(\wt{K}_v(J))$ on  $\wt{K}_v(J)$ lifts to an  action on $\wt{Y}_v(J)$, i.e.~a group homomorphism 
$\Aut^0(\wt{K}_v(J))\hra\Aut\left(\wt{\Sigma}_v(J)\right)$. 
\end{prp}
In order to prove the above proposition we recall  results of Rapagnetta~\cite{rapagnetta-og6} and Mongardi, Rapagnetta, 
Sacc\`a~\cite{monrapsac-og6}.  Let 
\begin{equation*}
S:=\Blow_{J[2]}(J)/\iota
\end{equation*}
be the quotient of the blow up of $J$ with center $J[2]$ by the involution lifting multiplication by $-1$ on $J$. Thus  
$S$ is  the (smooth) Kummer $K3$ surface associated to $J$. We have a commutative diagram
\begin{equation}\label{allevavisoni}
\xymatrix{   & \Blow_{J[2]}(J) \ar[dl]_{\mu}  \ar[dr]^{\nu} &   \\ 
  J \ar@{-->}[rr]^{2:1}  & & S }
\end{equation}
where $\mu$ is  the blow up map, $\nu$ is the quotient map, and the (rational) horizontal map identifies $a\in (J\setminus J[2])$ with $(-a)$. 
Let $h_S\in\NS(S)$ be the class such that  $\nu^{*}(h_S)=\mu^{*}(h)$. 
Let $w=(0,h_S,-1)\in \wt{H}(S;\ZZ)$, and let $M_w(S)$ be the moduli space of stable sheaves on $S$ with Mukai vector $w$ stable with respect to a $w$-suitable polarization (note that $h_S$ is not ample, it contracts the nodal curves of $S$). 
Then $\wt{Y}_v(J)$ is birational to $M_w(S)$. More precisely, we have a rational map $M_w(S) \dra Y_v(J)$
defined as follows. Let $[\cE]\in M_w(S)$ be a general point, so that $\cE=i_{D,*}(\xi)$ where $i_{D}\colon D\hra S$ is the inclusion of a smooth curve with 
$\cl(D)=h_S$ and $\xi$ is a line bundle on $D$ of degree $1$. Let $C:=\mu(\nu^{-1}(D))$. The restriction to $C$ of the rational map $J\dra S$ is an \'etale map $f_C\colon C\to D$ of degree $2$. Let $j_C\colon C\hra J$ be the inclusion map.
 Then $v(\mu^{*}(\cE))=v$, and since $C$ is irreducible, $j_{C,*}(f_C^{*}(\xi))$ is a stable sheaf. Lastly 
 $\Alb([j_{C,*}(f_C^{*}(\xi))])=0$. Hence we have rational map 
\begin{equation}\label{caterinabotti}
\begin{matrix}
 M_w(S) & \overset{\psi}{\dra} & \wt{K}_v(J) \\
 [i_{D,*}(\xi)] & \dra & [j_{C,*}(f_C^{*}(\xi))]
\end{matrix}
\end{equation}
Let   $\lambda_D$ be the  line bundle on $D$ (with trivial square) determined by the \'etale double cover $f_C\colon C\to D$. If 
$\xi'$ is a line bundle on $D$ then  $f_C^{*}(\xi')\cong f_C^{*}(\xi)$ if and only if $\xi'\cong\xi$ or 
$\xi'\cong\xi\otimes\lambda_D$. It follows that $\psi$ has degree $2$. In~\cite{monrapsac-og6}, see Lemma 5.2, one finds the proof that 
$M_w(S)$ is birational to $\wt{Y}_v(J)$ and that the double cover $\psi$ is identified (birationally) with the double cover 
$\wt{\epsilon}_{v}$. 

\begin{lmm}\label{lmm:piripicchio}
Let $D\subset S$ be a smooth curve with cohomology class $h_S$ (and hence integral), and let $C:=\mu(\nu^{-1}(D))$. Let 
$f_C\colon C\to D$ be the \'etale double cover given by the restriction of the horizontal map in~\eqref{allevavisoni}, and let $\Nm\colon \Pic(C)\to \Pic(D)$ be the corresponding norm map. If 
$[\xi]\in \wh{J}[2]$, then
\begin{equation}
f_C^{*}(\Nm(\xi))=\xi_{|C}.
\end{equation}
\end{lmm}
\begin{proof}
The map $\wh{J}\to\Pic^0(C)$ defined by restriction identifies $\wh{J}$ with the Prym variety $\Prym(f_C)$. By irreducibility of the moduli space of \'etale double covers of curves of a fixed genus (in our case genus $3$), it suffices to prove that 
the equality
\begin{equation}\label{allanorma}
f_C^{*}(\Nm(\xi))=\xi
\end{equation}
holds whenever we have an \'etale double cover $f_C\colon C\to D$ of a curve of genus $g$, and $[\xi]\in  \Prym(f_C)[2]$. We may also degenerate $C$ to a curve with separating nodes, and it suffices to prove the equality for such double covers. Choose $D=D_1\cup D_{g-1}$ where $D_1$ has genus $1$, $D_{g-1}$ has genus $g-1$, and $q_1\in D_1$ is glued to $q_{g-1}\in D_{g-1}$. Let $f_1\colon C_1\to D_1$ be an \'etale connected double cover, and let $f_1^{-1}(q_1)=\{p_1',p_1''\}$. Let $C'_{g-1}$ and $C''_{g-1}$ be copies of $D_{g-1}$, with 
$q'_{g-1}\in C'_{g-1}$ and $q''_{g-1}\in C''_{g-1}$ corresponding to $q_{g-1}$ respectively. Let 
$C:=C_1\cup C'_{g-1}\cup C''_{g-1}$ with $p_1'$ glued to $q'_{g-1}$ and $p''_1$ glued to $q''_{g-1}$. We have an obvious 
\'etale double cover $f_C\colon C\to D$. For this double cover $\Prym(f_C)=\Pic^0(D_{g-1})$, and one checks right away that the equality in~\eqref{allanorma} holds for all $[\xi]\in  \Prym(f_C)[2]=\Pic^0(D_{g-1})[2]$.
\end{proof}
\begin{proof}[Proof of Proposition~\ref{prp:sollevo}]
Let $x\in J[2]$. We define an automorphism $Y_v(J)\to Y_v(J)$ proceeding as follows.  The translation $\tau_x\colon J\to J$ defined by $x$ maps $J[2]$ to itself, hence it lifts to an automorphism of 
$\Blow_{J[2]}(J)$, and the latter descends to an automorphism $\ov{\tau}_x\colon S\to S$ because $\tau_x$ commutes with multiplication by $-1$. Since the generic sheaf parametrized by $M_w(S) $ is the push-forward of a line bundle on an irreducible curve on $S$, we have a birational map
\begin{equation}
\begin{matrix}
 M_w(S) & \overset{\ov{\lambda}_x}{\dra} &  M_w(S)  \\
 [\cE] & \dra & [\ov{\tau}_x^{*}(\cE)]
\end{matrix}
\end{equation}
Conjugating  with the birational map in~\eqref{caterinabotti}, we get a birational map 
\begin{equation}
\blambda_x\colon Y_v(J)\dra Y_v(J)
\end{equation}
 such that $\lambda_x\circ\wt{\epsilon}_v=\wt{\epsilon}_v\circ\blambda_x$, where 
 $\wt{\epsilon}_v\colon \wt{Y}_v(J)\to \wt{K}_v(J)$ is the double cover in~\eqref{doppelganger}. 
Since  $\wt{\epsilon}_v$ is a finite map, it follows that $\blambda_x$ is regular (the graph of $\blambda_x$ has a single point over any point of $ Y_v(J)$).

Now let $y\in \wh{J}[2]$. We define  an automorphism 
$Y_v(J)\to Y_v(J)$  proceeding as follows. Let $\cL_y\in\Pic^0(J)$ be the  line bundle corresponding to $y$.  Let 
$\ov{\cL}_y:=\Nm_{\nu}(\mu^{*}(\cL_y))$, where $\Nm_{\nu}\colon \Pic(\Blow_{J[2]}(J))\to \Pic(S)$ is the norm map defined by 
$\nu$. We claim that we have a birational map 
\begin{equation}
\begin{matrix}
 M_w(S) & \overset{\ov{\mu}_y}{\dra} &  M_w(S)  \\
 [\cE] & \dra & [\cE\otimes\ov{\cL}_y]
\end{matrix}
\end{equation}
In fact, let $ [\cE]$ be a general point of $M_w(S)$. Then   $\cE=i_{D,*}(\xi)$ where $D\subset S$ is a smooth curve with cohomology class $h_S$  (and hence integral). Let $C:=\mu_{*}(\nu^{*}(D))$. Then $\cE\otimes\ov{\cL}_y=i_{D,*}(\xi\otimes(\ov{\cL}_{y|D}))$, and $v(i_{D,*}(\xi\otimes(\ov{\cL}_{y|D})))=w$ because  $\deg(\ov{\cL}_{y|D})=0$ (since  $\deg(\cL_{y|C})=0$).  Moreover $\cE\otimes\ov{\cL}_y$ is stable because it is the push-forward of a line bundle on an integral curve. 
Conjugating  with the birational map in~\eqref{caterinabotti}, we get a birational map 
\begin{equation}
\bmu_y\colon Y_v(J)\dra Y_v(J)
\end{equation}
 such that $\mu_y\circ\wt{\epsilon}_v=\wt{\epsilon}_v\circ\bmu_x$. 
Arguing as in the case of $\blambda_x$ we get that $\bmu_y$ is regular. We 
get an embedding 
\begin{equation}
\begin{matrix}
J[2]\times\wh{J}[2] & \hra & \Aut(Y_v(J)) \\
(x,y) & \mapsto & \blambda_x\circ\bmu_y 
\end{matrix}
\end{equation}
The above homomorphism is a lift of the homomorphism in~\eqref{taipei}. In fact it is obvious that $\blambda_x$ lifts $\lambda_x$, while $\bmu_y$ lifts $\mu_y$ by Lemma~\ref{lmm:piripicchio}. 
\end{proof}
\subsection{The commutator pairing for line bundles of divisibility $2$}\label{subsec:theend}
\setcounter{equation}{0}
In the present subsection we prove the part of the statement of Theorem~\ref{thm:compairog6} that refers to line bundles of divisibility $2$. 
\begin{prp}\label{prp:comdivdue}
Let $X$ be a HK manifold of type OG6. If $L$ is a primitive line bundle on $X$ such that  $\divisore(L)=2$ then $e^L$ is trivial.  
\end{prp}
First we prove the above result for two special choices of line bundle $L$ on $\wt{K}_v(J)$, where $v=(0,2h,-2)$. 
\begin{prp}\label{prp:casispec}
Let $J$ and  $v=(0,2h,-2)$ be as in Subsection~\ref{subsec:og6rap}.  The commutator pairings of $\cL_v$ 
(notation as in 
Proposition~\ref{prp:ellevu}) and of $\cF_v$ (see~\eqref{autodecomp}) are trivial.
\end{prp}
\begin{proof}
The space of global sections of $\cL_v$ is one-dimensional by Proposition~\ref{prp:ellevu}, and hence 
the commutator pairing of $\cL_v$ is trivial by Remark~\ref{rmk:unasez}. 

By  Proposition~\ref{prp:sollevo}  the action of $\Aut^0(\wt{K}_v(J))$ on $\wt{K}_v(J)$ lifts to an action of 
$\Aut^0(\wt{K}_v(J))$ on $\wt{Y}_v(J)$. Since $\cF_v$ is the $(-1)$ eigenspace for the natural action   of the covering involution  of
$\wt{\epsilon}_v\colon \wt{Y}_v(J)\to\wt{K}_v(J)$ on $\wt{\epsilon}_{v,*}(\cO_{\wt{Y}_v(J)})$, it follows that the action of 
$\Aut^0(\wt{K}_v(J))$ on $\wt{K}_v(J)$ lifts to an action of 
$\Aut^0(\wt{K}_v(J))$ on $\cF_v$. By Remark~\ref{rmk:solltriv} we get  that the commutator pairings of   $\cF_v$ is  trivial.
\end{proof}
Let $\Lambda_{\OG}$ be the lattice defined in~\eqref{retog6}. 
\begin{prp}\label{prp:solitoeichler}
Let $a,b\in \Lambda_{\OG}$ be primitive elements. Then $a,b$ belong to the same $\Ort^{+}(\Lambda_{\OG})$-orbit if and only if $(a,a)=(b,b)$ and either Item~(I), or Item~(II), or Item~(III) of Proposition~\ref{prp:vettog6} holds both for $a$ and $b$.
\end{prp}
\begin{proof}
The non trivial implication is \lq\lq if either Item~(I), or Item~(II), or Item~(III)  of Proposition~\ref{prp:vettog6} holds both for $a$ and $b$,  then $a,b$ belong to the same $\Ort^{+}(\Lambda_{\OG})$-orbit\rq\rq. This result holds  by Eichler's Criterion,  see Item~(i) of Proposition~3.3 in~\cite{ghs-abelianisation}.
\end{proof}
\begin{lmm}\label{lmm:alice}
Let  $X$ be a HK manifold of type $\OG$ carrying   a primitive line bundle $L$ of divisibility $2$, and hence (see Proposition~\ref{prp:vettog6}) either
$q_X(L)\equiv -2\pmod{8}$ or $q_X(L)\equiv -4\pmod{8}$. Let $d\in\ZZ$ be such that
\begin{equation}
q_X(L)=
\begin{cases}
-2+8d & \text{if $q_X(L)\equiv -2\pmod{8}$,} \\
-4+8d & \text{if $q_X(L)\equiv -4\pmod{8}$.}
\end{cases}
\end{equation}
There exist a family  $f\colon \cX\to T$ of HK manifolds  over a connected base $T$, points $t_0,t_1\in T$ and a line bundle $\cL$ on $\cX$ with the following properties:
\begin{enumerate}
\item[(a)]
 The fiber $X_{t_0}:=f^{-1}(t_0)$ is isomorphic to $X$, and  
  $L_{t_0}:=\cL_{|X_{t_0}}$  is isomorphic to $L$. 
\item[(b)]
 The fiber $X_{t_1}:=f^{-1}(t_1)$ is birational to  $\wt{K}_v(J)$ where $J$ is a principally polarized abelian surface as in Subsection~\ref{subsec:og6rap}, $v=(0,2h,-2)$ and 
\begin{equation}
c_1(L_{t_0})=
\begin{cases}
\pm(\alpha_v+2(-d\alpha_v+d\gamma_v) )& \text{if $q_X(L)\equiv -2\pmod{8}$,} \\
\pm(\beta_v-\alpha_v+2(d\alpha_v+d\gamma_v) )& \text{if $q_X(L)\equiv -4\pmod{8}$.}
\end{cases}
\end{equation}
(Notation as in Subsection~\ref{subsec:og6rap}.)
\end{enumerate}
\end{lmm}
\begin{proof}
First note that 
\begin{equation}\label{quadrato1}
(\alpha_v+2(-d\alpha_v+d\gamma_v) ,\alpha_v+2(-d\alpha_v+d\gamma_v) )=-2+8d,
\end{equation}
\begin{equation}\label{quadrato2}
( \beta_v-\alpha_v+2(d\alpha_v+d\gamma_v), \beta_v-\alpha_v+2(d\alpha_v+d\gamma_v))=-4+8d,
\end{equation}
and that both $\alpha_v+2(-d\alpha_v+d\gamma_v) $ and $\beta_v-\alpha_v+2(d\alpha_v+d\gamma_v) $ have divisibility $2$. 

The lemma is a standard consequence  of Verbitsky's global Torelli Theorem, the 
monodromy computations of Mongardi-Rapagnetta and Proposition~\ref{prp:solitoeichler}. We quickly go over the argument. 

Let $X$ be a HK manifold of type OG6. By Theorem~1.4 in~\cite{mon-rap-mon-og6} the monodromy of $H^2(X;\ZZ)$ is the group $\Ort^{+}(H^2(X;\ZZ))$ of isometries preserving a (continuous) choice of orientations of the maximal positive definite subspaces of $H^2(X;\RR)$ (an index $2$ subgroup of the orthogonal group $\Ort(H^2(X;\ZZ))$). 

Let $\gM$ be the moduli space of marked HK manifolds of type OG6. By the result on monodromy quoted above, there are exactly $2$ connected components of the (non Hausdorff) complex manifold $\gM$, interchanged by mapping
 $[(X,\varphi)]$  (here $\varphi\colon H^2(X;\ZZ)\overset{\sim}{\lra}\Lambda_{\OG}$ is an isometry) to 
 $[(X,-\varphi)]$. Let $\gM^0$ be one of the two connected components of $\gM$. 

Let $\cD\subset\PP(\Lambda_{\OG}\otimes\CC)$ be the period domain, and let $\cP\colon\gM^0\to\cD$ be the period map. Suppose that $[(X_1,\varphi_1)]$ and $[(X_2,\varphi_2)]$ belong to the same fiber of $\cP$; then $X_1$ is birational to $X_2$ by Verbitsky's global Torelli Theorem~\cite{verbitskytor}, and $[(X_1,\varphi_1)]$, $[(X_2,\varphi_2)]$ are non separated points in $\gM^0$. There is a Hausdorffization $\gM^0\to\gM^0_{\text H}$  
with an induced period map $\cP\colon\gM^0_{\text H}\to\cD$ which is an isomorphism of complex manifolds.

Let $X$ be a HK manifold of type OG6 carrying  a line bundle $L$. Let 
$\varphi\colon H^2(X;\ZZ)\overset{\sim}{\to}\Lambda_{\OG}$ be an isometry, and suppose that  $[(X,\varphi)]\in\gM^0$. Then 
$\cP(X,\varphi)\in \varphi(c_1(L))^{\bot}$. Conversely, if $a\in\Lambda_{\OG}$, and $[(X,\varphi)]\in\gM^0$ is such that $\cP(X,\varphi)\in a^{\bot}$, then $\varphi^{-1}(a)$ is the first Chern class of a line bundle.

Let  $a,b\in\Lambda_{\OG}$ be primitive elements of divisibility $2$ such that either $(a,a)=(b,b)=-2+8d$ or 
$(a,a)=(b,b)=-4+8d$. 
By Proposition~\ref{prp:solitoeichler} there exists $g\in\Ort^{+}(\Lambda_{\OG})$ such that $g(a)=b$ or $g(a)=-b$. (Note that $\Ort(\Lambda_{\OG})$ is the direct product of $\Ort^{+}(\Lambda_{\OG})$ and $\la -\Id\ra$). Since the hyperplane $a^{\bot}\cap\cD$ for a fixed $a$ as above is connected, the lemma follows
by the monodromy result. (Recall the equalities in~\eqref{quadrato1} and~\eqref{quadrato2} and the sentence following those equations.) 
\end{proof}
\begin{proof}[Proof of Proposition~\ref{prp:comdivdue}]
By Lemma~\ref{lmm:alice} and by invariance of the commutator pairing under deformation (see Remark~\ref{rmk:variacom}) and birational maps (see Remark~\ref{rmk:stessocomm}), it suffices to prove that the commutator pairing is trivial for a line bundle $L$ on $\wt{K}_v(J)$ (notation as in Subsection~\ref{subsec:og6rap}, in particular $v=(0,2h,-2)$) such that 
$c_1(L)=\alpha_v+2(-d\alpha_v+d\gamma_v)$ or $c_1(L)=\beta_v-\alpha_v+2(d\alpha_v+d\gamma_v)$. If the former holds there exists a line bundle $\xi$ on $\wt{K}_v(J)$ such that $L=\cF^{-1}_v\otimes \xi^{\otimes 2}$, and if the latter holds  there exists a line bundle $\xi$ on $\wt{K}_v(J)$  such that $L=\cL_v\otimes \xi^{\otimes 2}$. (Here
 $\cF_v$ is as in~\eqref{autodecomp}, $\cL_v$ is as in Proposition~\ref{prp:ellevu}.) 
It follows that $e^L$ is trivial by Proposition~\ref{prp:casispec} and because the square of a line bundle on a manifold $X$ of type OG6 has trivial commutator pairing (multiplication by $2$ kills every element of  $T(X)$).
\end{proof}
\subsection{Varieties of type OG6 corresponding to $v=(0,2h,0)$}\label{subsec:og6nostre}
\setcounter{equation}{0}
We let $J$ be as in Subsection~\ref{subsec:og6rap}, in particular we assume that~\eqref{picarduno} holds. One can repeat all the constructions of that subsection with the Mukai vector $(0,2h,-2)$ replaced by  $v=(0,2h,0)$. 
 Thus we have $M_v(J)$, $\wt{M}_v(J)$ and, upon choosing  a symmetric principal polarization
$\Theta_J$  of $J$, we have 
$K_v(J)\subset M_v(J)$ and $ \wt{K}_v(J)\subset \wt{M}_v(J)$.
We analyze  $\wt{K}_v(J)$  in order to prove the validity of Theorem~\ref{thm:compairog6} for line bundles of divisibility $1$. 

By the isomorphism in~\eqref{reticolog6}, the N\'eron-Severi group of $\wt{K}_v(J)$ is freely generated by 
$\theta_v(1,0,0)$, $\theta_v(0,0,1)$ and $\alpha_v$. We  provide geometric descriptions of each of these classes. 
Proceeding exactly as in Subsection~\ref{subsec:og6rap} one defines 
 a  map 
\begin{equation}\label{emmastone}
\begin{matrix}
K_v(J) & \overset{\pi_v}{\lra} & |2\Theta_J|\cong\PP^3\\
[\cF] & \mapsto & \Supp_{\det}(\cF)
\end{matrix}
\end{equation}
Composing with the desingularization map one gets the Lagrangian fibration
\begin{equation}\label{fibrlagr}
\wt{\pi}_v\colon \wt{K}_v(J) \to |2\Theta_J|\cong\PP^3
\end{equation}
 Let 
 $\Lambda_v:=c_1\left(\pi_v^*\cO_{\PP^3}(1)\right)$, and let $\wt{\Lambda}_v$ be the pull-back  to $\wt{K}_v(J)$ of  $\Lambda_v$. 
Then
\begin{equation}
\wt{\theta}_v(0,0,-1)=\wt{\Lambda}_v
\end{equation}
by the same result quoted in Subsection~\ref{subsec:og6rap}.
Next,  let $\Theta_v\subset K_v(J)$ be the (reduced) divisor parametrizing sheaves $\cF$ such that $h^0(\cF)>0$ (note: this is a divisor because $\chi(J,\cF)=0$ for $[\cF]\in K_v(J)$), and  let 
$\wt{\Theta}_v$ be its pull-back  to $\wt{K}_v(J)$.
Since $\Theta_v$  is the zero locus of the canonical section of the determinant line bundle on $K_v(J)$, we have 
$\theta_v(1,0,0)=\cl(\Theta_v)$. Thus
\begin{equation}\label{tetateta}
\wt{\theta}_v(1,0,0)=\cl(\wt{\Theta}_v).
\end{equation}
Lastly $2\alpha_v=\cl(\wt{\Sigma}_v)$ where $\wt{\Sigma}_v(J)\subset \wt{K}_v(J)$ is the exceptional divisor of
the desingularization map $\wt{K}_v(J)\to K_v(J)$, see Subsetion~\ref{subsec:varog6}.  The conclusion is that 
 we have a direct sum decomposition
\begin{equation}\label{nersev}
\NS(\wt{K}_v(J))=\ZZ\cl(\wt{\Theta}_v)\oplus \ZZ\wt{\Lambda}_v \oplus  \ZZ\alpha_v.
\end{equation}
We record here the equalities (recall that $\wt{\theta}_v$ is an isometry of lattices)
\begin{equation}\label{pianiper}
q(\cl(\wt{\Theta}_v),\cl(\wt{\Theta}_v))=0,\quad q(\cl(\wt{\Theta}_v),\wt{\Lambda}_v)=1,\quad
q(\wt{\Lambda}_v,\wt{\Lambda}_v)=0.
\end{equation}
\subsection{Fourier-Mukai transform}\label{sec:foumuk}
\setcounter{equation}{0}
Let $\wh{J}=\Pic^0(J)$ be the dual of $J$, where $J$ is as in Subsection~\ref{subsec:og6nostre}, in particular~\eqref{picarduno} holds. Hence $\NS(\wh{J})=\ZZ \wh{h}$, where $\wh{h}$ is the unique principal polarization of $\wh{J}$. Since $\wh{J}$ is isomorphic to $J$ all the definitions and considerations of Subsection~\ref{subsec:og6nostre} apply to $\wh{J}$. More precisely, let  $\wh{v}:=(0,2\wh{h},0)$. We have the moduli space 
$M_{\wh{v}}(\wh{J},\wh{h})$, which we denote by $M_{\wh{v}}(\wh{J})$, and, upon choosing  a symmetric principal polarization $\Theta_{\wh{J}}$ of $\wh{J}$,  the singular symplectic variety
 $K_{\wh{v}}(\wh{J})$ and its  HK desingularization $\wt{K}_{\wh{v}}(\wh{J})$.  On $\wt{K}_{\wh{v}}(\wh{J})$ we have the divisor  
 $\wt{\Theta}_{\wh{v}}$  and the divisor classes    
$\wt{\Lambda}_{\wh{v}}$, $\alpha_{\wh{v}}$. 

Of course an isomorphism $J\overset{\sim}{\lra} \wh{J}$ mapping $\Theta_{J}$  to  $\Theta_{\wh{J}}$ 
determines an isomorphisms $\wt{K}_v(J) \overset{\sim}{\lra} \wt{K}_{\wh{v}}(\wh{J})$,
but we are interested  in a different birational mapping $\wt{K}_v(J) \dra \wt{K}_{\wh{v}}(\wh{J})$, given by a Fourier-Mukai transform.  More precisely, let 
$\cP$ be the Poincar\`e line bundle on $J\times\wh{J}$, and let $\sF\sM\colon D^{b}(J)\to D^{b}(\wh{J})$ be the Fourier-Mukai transform with kernel $\cP$. 

Let
\begin{equation}\label{pizzino}
U_v(J):=\{[\cF]\in K_v(J) \mid \text{$\cF$ polystable, $H^0(J,\cF\otimes\cL_y)=0$ for general $y\in\wh{J}$}\}.
\end{equation}
\begin{lmm}\label{lmm:indideb}
Keeping notation as above, the following properties hold.
\begin{enumerate}
\item[(A)]
Every  sheaf $\cF$ parametrized by a point of $U_v(J)$ satisfies $\WIT_1$ with respect to the Poincar\`e line bundle $\cP$, and hence  $\sF\sM(\cF)$ is a  sheaf shifted by $[-1]$.
\item[(B)]
The complement of $U_v(J)$ has codimension at least $2$ in $K_v(J)$. 
\item[(C)]
The set  $\Sigma_v(J)$ of strictly semistable intersects $U_v(J)$.
\item[(D)]
 If  $[\cF]\in (U_v(J)\setminus\Sigma_v(J))$,  the sheaf $\sF\sM(\cF)[1]$ is $\wh{h}$-stable. 
\item[(E)]
 If  $[\cF]\in \Sigma_v(J)\cap U_v(J)$,  the sheaf $\sF\sM(\cF)[1]$ is strictly $\wh{h}$-semistable. 
\end{enumerate}
\end{lmm}
\begin{proof}
Item~(A) holds for $U_v(J)$ defined as in~\eqref{pizzino} because sheaves parametrized by $K_v(J)$ have one dimensional support. 

We prove that Item~(B) holds.  Let $C\in|2\Theta_J|$ be a smooth curve. We claim that the complement of  
$U_v(J)\cap\pi_v^{-1}(C)$ (see~\eqref{emmastone}) in $\pi_v^{-1}(C)$ has codimension at least $2$. In fact, let $f_C\colon C\to D$ be the quotient map for the $\ZZ/(2)$-action defined by multiplication by $(-1)$ in $J$. Thus $D$ is a smooth plane section of the Kummer surface associated to $J$ (see~\eqref{allevavisoni}). We have $f_{C,*}\cO_C=\cO_D\oplus\lambda_D$ where $\lambda_D$ is a non trivial line bundle whose square is trivial. Let $[\cF]\in\pi_v^{-1}(C)$, i.e.~$\cF=j_{C,*}(\xi)$ where $\xi$ is a  line bundle of degree $4$ on $C$ such that (here Chern classes are in the Chow ring)
\begin{equation}\label{sommazero}
0=\sigma(c_2(j_{C,*}(\xi)))=\sigma(C\cdot C-j_{C,*}(c_1(\xi)))=-\sigma(j_{C,*}(c_1(\xi)).
\end{equation}
The equalities in~\eqref{sommazero} give  that $\xi=f_C^{*}(\ov{\xi})$, where $\ov{\xi}$ is a line bundle on $D$ of degree $2$. Let $y\in\wh{J}[2]$, and let 
$\cL_y$ be the corresponding line bundle on $J$; then $\cL_{y|C}$ is invariant under multiplication by $(-1)$, and hence there exists a line bundle $\alpha_y$ on $D$ such that $\cL_{y|C}\cong f_C^{*}(\alpha_y)$. Hence
\begin{equation}\label{ludovisi}
H^0(C,\xi\otimes(\cL_{y|C}))=H^0(C,f_C^{*}(\ov{\xi}\otimes\alpha_y))=H^0(D,(\ov{\xi}\otimes\alpha_y)\oplus
(\ov{\xi}\otimes\alpha_y\otimes\lambda_D)).
\end{equation}
Let $\Theta_D\subset\Pic^2(D)$ be the natural theta divisor. The equalities in~\eqref{ludovisi} give that 
\begin{equation}
\pi_v^{-1}(C)\setminus U_v(J)\subset \bigcap\limits_{y\in\wh{J}[2]}\left((\Theta_D+[\alpha_y])+(\Theta_D+[\alpha_y\otimes\lambda_D])\right).
\end{equation}
Since $\Theta_D$ is irreducible, it follows that $\pi_v^{-1}(C)\setminus U_v(J)$ has codimension at least $2$ in 
$\pi_v^{-1}(C)$. In order to finish the proof of Item~(B) it suffices to prove the complement of $U_v(J)$ does not contain a  divisor mapping to one of the two irreducible components of the discriminant hypersurface in $|2\Theta_J|$, i.e.~the closure of the locus parametrizing curves $C\in|2\Theta_J|$ with a single node at a point of $J[2]$, and the locus parametrizing curves $C=\tau_x^{*}(\Theta_J)+\tau_{-x}^{*}(\Theta_J)$ where $x\in J$. This is easy, we leave details to the reader.

Items~(C), (D) and (E) are straighforward.
\end{proof}
Let $[\cF]\in U_v(J)$. Since 
$v(\sF\sM(\cF)[1])=(0,2\wh{h},0)$, the Fourier-Mukai transform defines a regular map $U_v(J)\to M_{\wh{v}}(\wh{J}) $. The image lies  in an Albanese fiber (we mean the map in~\eqref{franbei} with $J,v$ replaced by $\wh{J},\wh{v}$ respectively)) because every map from $K_v(J)$ to an abelian variety is constant. By considering the image of points in $\Sigma_v(J)\cap U_v(J)$ we get that it lands in 
$K_{\wh{v}}(\wh{J})$. By Items~(C) and~(E) in Lemma~\ref{lmm:indideb} we get a 
 a rational mapping 
\begin{equation*}
\varphi\colon \wt{K}_v(J) \dra \wt{K}_{\wh{v}}(\wh{J}).
\end{equation*}
The map $\varphi$ is birational; the inverse is given by the \lq\lq reverse\rq\rq\ Fourier Mukai transform.
\begin{prp}
Keeping notation as above, we have
\begin{equation}\label{scambio}
\varphi^{*}(\cl(\wt{\Theta}_{\wh{v}}))=\wt{\Lambda}_v,\quad 
\varphi^{*}(\wt{\Lambda}_{\wh{v}})=\cl(\wt{\Theta}_v),\quad \varphi^{*}(\alpha_{\wh{v}})=\alpha_v.
\end{equation}
\end{prp}
\begin{proof}
Since  the open subset
$U_v(J)$ has complement of codimension at least $2$ in $K_v(J)$, one has the equality  
$\varphi^{*}(\theta_{\wh{v}}(r,xh,s))=\theta_v(-s,xh,-r)$ by a well-known computation, see Lemma~3.1 in~\cite{Yos:moduli}.
This gives the first two equalities in~\eqref{scambio}. The third equality holds by Items~(C) and~(E) in Lemma~\ref{lmm:indideb}.
\end{proof}
\subsection{Translations of $\wt{K}_v(J)$ for $v=(0,2h,0)$}\label{sec:mongwand}
\setcounter{equation}{0}
The definition of the  action of $J[2]\times\wh{J}[2]$ on $\wt{K}_v(J)$ for $v=(0,2h,-2)$ extends verbatim to give  
an action of $J[2]\times\wh{J}[2]$ on $\wt{K}_v(J)$ for $v=(0,2h,0)$.

The proof that the group of automorphisms of a HK manifold of type OG6 is isomorphic to $(\ZZ/(2))^8$ was achieved by examining   $\wt{K}_v(J)$  for $v=(0,2h,0)$. In fact Mongardi and Wandel proved the following result.
\begin{thm}[Thm.~5.2 in~\cite{mon-wand}]\label{thm:transog6}
Keep notation as above, in particular  $v=(0,2h,0)$. Then the map $J[2]\times \wh{J}[2]\hra \Aut^0(\wt{K}_v(J))$ is an isomorphism of groups.
\end{thm}
Here
we recall one element in the proof of  Theorem~\ref{thm:transog6} because this gives us a chance to
correct a statement in~\cite{mon-wand}, and also to show that the error does not affect the truth of  Theorem~\ref{thm:transog6}. 
  Lemma~5.4 in op.cit.~states that the linear system $|\wt{\Theta}_v|$ consists of the single divisor $\wt{\Theta}_v$. This statement is wrong. In fact, by the second equality in~\eqref{scambio}, the pull-back $\varphi^{*}$ gives an isomorphism 
\begin{equation}
\varphi^{*}\colon H^0(K_{\wh{v}}(\wh{J}),\wt{\Lambda}_{\wh{v}}) \overset{\sim}{\lra} H^0(K_v(J),\wt{\Theta}_v),
\end{equation}
and     $h^0(K_{\wh{v}}(\wh{J}),\wt{\Lambda}_{\wh{v}})=h^0(K_{v}(J),\wt{\Lambda}_{v})=4$, see~\eqref{emmastone}.
(The reason why 
the proof of Lemma~5.4 in~\cite{mon-wand} is wrong is that the
divisor $\wt{\Theta}_v$ has an open dense subset which is birational to a $\PP^1$-fibration, but it fails to be  
 normal, and hence the hypothesis of  Lemma~2.5 in op.cit.~is not satisfied by 
$\wt{\Theta}_v$.)

Lastly, although Lemma~5.4 in~\cite{mon-wand} is wrong,  the statement of Theorem~\ref{thm:transog6} is valid. In fact 
Mongardi-Wandel prove that if $g\in\Aut^0(\wt{K}_v(J))$ then there exists  $\lambda_x$ such that $\lambda_x\circ g$
 maps each (Lagrangian) fiber of $\wt{\pi}$ to itself. Since $\lambda_x\circ g$ maps the line bundle 
$\cO_{\wt{K}_v(J)}(\wt{\Theta}_v)$ to itself, the restriction of $\lambda_x\circ g$ to  a smooth Lagrangian fiber $A_t$ is equal to the restriction of a translation by an element $y\in \im(\wh{J}[2]\to \wh{A}_t[2])$ (this holds because the  polarization of $\wt{\Theta}_v$ on $A_t$ has elementary divisors $(1,2,2)$).  Clearly $y$ is independent of $t$, i.e.~we have $g=\lambda_{-x}\circ \mu_y$.

\subsection{The commutator pairing for line bundles of divisibility $1$}\label{sec:divuno}
\setcounter{equation}{0}
In the present subsection we prove the part of the statement of Theorem~\ref{thm:compairog6} that refers to line bundles of divisibility $1$.
\begin{prp}\label{prp:comdivuno}
Let $X$ be a HK manifold of type $\OG$. Suppose that $L$ is a primitive line bundle on $X$ such that  $\divisore(L)=1$.  
\begin{enumerate}
\item
If  $q_X(L)$ is not divisible by $4$ then  $e^L$ is non degenerate.
\item
If  $q_X(L)$ is divisible by $4$ then  $\wh{T}(X)/\im(E_L)\cong (\ZZ/(2))^4$.
\end{enumerate}
\end{prp}
Before proving Proposition~\ref{prp:comdivuno} we describe the commutator pairing of the line bundles  on $\wt{K}_v(J)$ given by $\pi^{*}(\cO_{\PP^3}(1))$ and
$\cO_{\wt{K}_v(J)}(\wt{\Theta})$. 
\begin{prp}\label{prp:thaddeus}
Keep notation as in Subsections~\ref{subsec:og6nostre}, \ref{sec:foumuk} and~\ref{sec:mongwand}.
The commutator pairing of $\pi_v^{*}(\cO_{\PP^3}(1))$ has kernel $\wh{J}[2]$ and is non degenerate on $J[2]$.
Similarly, the commutator pairing of $\cO_{\wt{K}_v(J)}(\wt{\Theta}_v)$ has kernel $J[2]$ and is non degenerate on 
$\wh{J}[2]$.  
\end{prp}
\begin{proof}
Let us prove the first statement of the lemma. By~\eqref{emmastone} we have an identification
\begin{equation}
H^0(\wt{K}_v,\pi_v^{*}(\cO_{\PP^3}(1))\cong H^0(J, \cO_J(2\Theta_J))^{\vee}.
\end{equation}
Since $\wh{J}[2]$ acts trivially  on $\PP(H^0(J, \cO_J(2\Theta_J))^{\vee})$, it follows that  $\wh{J}[2]$ is in the  kernel of the commutator pairing of 
$\pi_v^{*}(\cO_{\PP^3}(1))$, see Remark~\ref{rmk:unasez}. 
On the other hand, let  $\cG( J[2])<\cG(\pi_v^{*}(\cO_{\PP^3}(1)))$ be the inverse image of $J[2]$ under the natural homomorphism 
$\cG(\pi_v^{*}(\cO_{\PP^3}(1)))\to \Aut^0(\wt{K}_v(J))$. Since the action of $\cG( J[2])$ on $ H^0(J, \cO_J(2\Theta_J))^{\vee}$  is identified with the action of the theta group of $\cO_J(2\Theta_J)$ on $H^0(J, \cO_J(2\Theta_J))^{\vee}$, which is the Schr\"odinger representation, 
it follows that the commutator pairing of 
$\pi_v^{*}(\cO_{\PP^3}(1))$  is non degenerate on $J[2]$. 

The second statement of the lemma follows from the first one (that we have just proved) because of the Fourier Mukai transform discussed in Subsection~\ref{sec:foumuk}. In fact, the switching in Equation~\eqref{scambio} gives a natural isomorphism
\begin{equation}\label{lovelynight}
\varphi^{*}\colon H^0(\wt{K}_{\hat{v}}(\wh{J}),\pi_{\hat{v}}^{*}(\cO_{\PP^3}(1))\overset{\sim}{\lra} 
H^0(\wt{K}_{v}(J), \cO_{\wt{K}_v(J)}(\wt{\Theta}_v))^{\vee}.
\end{equation}
Since the space of sections is not trivial, the commutator pairing can be recovered by
the action of the theta group on the space of sections. The action as just been described, provided one keeps in mind that 
$J=\wh{\wh{J}}$. 

\end{proof}
\begin{crl}\label{crl:casoab}
With notation as above, let $L$ be a primitive line bundle on $\wt{K}_v(J)$ such that 
\begin{equation}
c_1(L)=a \wt{\Lambda}_v+b \cl(\wt{\Theta}_v).
\end{equation}
\begin{enumerate}
\item
If  $a,b$ are both odd, i.e.~$q(L)$ is not divisible by $4$ (see~\eqref{pianiper}), then  $e^L$ is non degenerate.
\item
If one among $a,b$ is even, i.e.~$q(L)$ is divisible by $4$, then  $\wh{T}(\wt{K}_v(J))/\im(E_L)$ is isomorphic to  $(\ZZ/(2))^4$.
\end{enumerate}
\end{crl}
\begin{proof}
The result  follows at once from Proposition~\ref{prp:thaddeus}, multiplicativity of the commutator pairing, and the fact that multiplication by $2$ kills every 
element of $T(\wt{K}_v(J))$. 
\end{proof}
The result below is analogous to Lemma~\ref{lmm:alice}. We omit the proof because   it is analogous to the proof of 
Lemma~\ref{lmm:alice}.
\begin{lmm}\label{lmm:circeo}
Let  $X$ be a HK manifold of type $\OG$ carrying   a primitive line bundle $L$ of divisibility $1$. There exist a family  $f\colon \cX\to T$ of HK manifolds  over a connected base $T$, points $t_0,t_1\in T$ and a line bundle $\cL$ on $\cX$ with the following properties:
\begin{enumerate}
\item[(a)]
 The fiber $X_{t_0}:=f^{-1}(t_0)$ is isomorphic to $X$, and  
  $L_{t_0}:=\cL_{|X_{t_0}}$  is isomorphic to $L$. 
\item[(b)]
 The fiber $X_{t_1}:=f^{-1}(t_1)$ is birational to  $\wt{K}_v(J)$ where $J$ is a principally polarized abelian surface as in Subsection~\ref{subsec:og6nostre}, $v=(0,2h,o)$ and 
\begin{equation}
c_1(L_{t_0})=a \wt{\Lambda}_v+b \cl(\wt{\Theta}_v)
\end{equation}
for suitable $a,b$.
\end{enumerate}
\end{lmm}
\begin{proof}[Proof of Proposition~\ref{prp:comdivuno}]
The proposition follows at once from Corollary~\ref{crl:casoab}, Lemma~\ref{lmm:circeo}, and invariance of the commutator pairing under deformation (see Remark~\ref{rmk:variacom}) and birational maps (see Remark~\ref{rmk:stessocomm}).
\end{proof}
\section{The commutator pairing for certain rank $4$ vector bundles}\label{sec:rango4}
\setcounter{equation}{0}
\subsection{The computation}
\setcounter{equation}{0}
We recall the setting of Theorem~\ref{thm:heisenrg4}. Let  $e$ be a positive integer such that $e\equiv -6 \pmod{16}$, and  
let $(M,h)$ be a general polarized HK fourfold of Kummer type with $q_M(h)=e$ and the divisibility of $h$ is $2$. 
In~\cite{og:modonkum} we have shown that there exists a slope stable rank $4$ vector bundle $\cF$ on $M$ such that
\begin{equation}
\det\cF\cong\cO_M(h),\quad \Delta(\cF):=8c_2(\cF)-3c_1(\cF)^2=c_2(M).
\end{equation}
We have also proved that $g^{*}(\cF)\cong\cF$ for every $g\in\Aut^0(M)$, and
hence the theta group $\cG(\cF)$ is defined.
\begin{thm}\label{thm:giulia}
Keeping notation as above, we have
\begin{equation}\label{maryclaire}
\wh{T}(M)/\im(E_{\cF})\cong 
\begin{cases}
0 & \text{if $q_M(c_1(\cF))$ is not divisible by $3$,} \\
(\ZZ/(3))^2 & \text{if $q_M(c_1(\cF))$ is  divisible by $3$.}
\end{cases}
\end{equation}
\end{thm}
\begin{proof}
In~\cite{og:modonkum} we obtained the sheaves $\cF$ by deformation of certain stable modular sheaves  on the generalized Kummer fourfold $K_2(A)$ associated to   an abelian surface $A$. 
By invariance under deformation, it suffices to prove that~\eqref{maryclaire} holds for $M=K_2(A)$. We recall the definition of the modular sheaves on  $K_2(A)$. Let $f\colon B\to A$ be a degree $2$ homomorphism of  abelian surfaces. By mapping a general $[Z] \in K_{2}(B)$ to $[f(Z)]\in K_{2}(A)$ one defines a rational map
$\rho\colon K_{2}(B) \dra K_{2}(A)$ whose indeterminacy locus is equal to
\begin{equation}\label{vuenne}
V(f):=\{[Z]\in K_2(B) \mid \ell(f(Z))<\ell(Z)=2\}.
\end{equation}
The  blow up $\nu\colon X\to K_2(B)$  of $V(f)$ resolves the indeterminacies of  $\rho$.  
Thus we have  
the commutative diagram 
\begin{equation}\label{scivolo}
\xymatrix{   & X \ar[dl]_{\nu}  \ar[dr]^{\wt{\rho}} &   \\ 
  K_2(B)  \ar@{-->}[rr]^{\rho} & & K_2(A) }
\end{equation}
For   a line bundle $\cL$ on $X$, we let 
$\cE(\cL):=\wt{\rho}_{*}(\cL)$. Now assume that $\cL=\nu^{*}(\ov{\cL})$, where $\ov{\cL}$ is a line bundle on $K_2(B)$ such that $c_1(\ov{\cL})=\mu_B(\omega_B)$, where $\mu_B\colon H^2(B)\to H^2(K_2(B))$ is the symmetrization map. We have
\begin{equation}
c_1(\cE(\cL))=2\mu_A(\omega_A)-\delta_2(A)
\end{equation}
where $\omega_A:=f_{*}(\omega_B)$. 
There are assumptions on $B$ and $\omega_B$ which guarantee that $\cE(\cL)$ is a stable rank $4$ vector bundle, that for a  general deformation of $K_2(A) $ which keeps the class of $2\mu_A(\omega_A)-\delta_2(A)$ of type $(1,1)$, the class  is ample of square $e=16a-6$ and divisibility $2$, that  $\cE(\cL)$ 
 extends to a (stable) vector bundle on such a general deformation, and that these vector bundles are stabilized by 
the group of automorphisms acting trivially on $H^2$, see Section 6 in~\cite{og:modonkum}. 
 In particular, among the relevant assumptions we have that $\omega_B$ is a primitive class and 
\begin{equation}\label{enaena}
\omega_B\cdot\omega_B=2a,\quad e=16a-6.
\end{equation}
 The commutative diagram in~\eqref{scivolo} gives an identification
\begin{equation}\label{lebatterie}
H^0(K_2(B),\ov{\cL}) \cong H^0(K_2(A),\cE(\cL)).
\end{equation}
Note that $\ov{\cL}$ is big and nef, hence $H^0(K_2(B),\ov{\cL}) $ is non trivial, in fact the formula in~\eqref{hiriro} gives that
\begin{equation}\label{vittoria}
3\cdot{a+2\choose 2}=h^0(K_2(B),\ov{\cL}) =h^0(K_2(A),\cE(\cL)).
\end{equation}
The isomorphism $B[3]\overset{\sim}{\lra} A[3]$  induced by the  homomorphism $f\colon B\to A$ of degree $2$ defines an isomorphism
\begin{equation}
T(K_2(B))=B[3]\overset{\sim}{\lra} A[3]=T(K_2(A))
\end{equation}
 which commutes with the actions of the two groups on the two sides of the isomorphism in~\eqref{lebatterie}. Since we may read off the commutator pairing from the action of the  theta group on  $H^0(K_2(A),\cE(\cL))$ (because it is non trivial), Theorem~\ref{thm:compairkum} gives that the commutator pairing is non degenerate if $a$ is not divisible by $3$, and that $\wh{T}(M)/\im(E_{\cF})$ is isomorphic to $(\ZZ/(3))^2$ if $a$ is  divisible by $3$. By  the last equality in~\eqref{enaena}, this proves the result. 
\end{proof}
\subsection{An example}
\setcounter{equation}{0}
We briefly discuss the first instance of non degenerate commutator pairing of Theorem~\ref{thm:giulia}. Let $(M,h)$ be a general polarized $4$ dimensional HK variety of Kummer type with $q_M(h)=10$ and  divisibility of $h$ equal to $2$. Let 
$\cF$ be a stable rank $4$ vector bundle on $M$ as in  Theorem~\ref{thm:giulia}. Then the commutator pairing $e^{\cF}$ is non degenerate, and hence $H^0(M,\cF)$, which has dimension $9$ by~\eqref{vittoria}, is the  Schr\"odinger representation of $T(M)\cong (\ZZ/(3))^4$. Let
\begin{equation}
\PP(\cF^{\vee})\overset{\phi_{\cF}}{\dra} \PP(H^0(M,\cF)^{\vee})
\end{equation}
be the natural map. If the general fiber of $\phi_{\cF}$ is finite, then the image of $\phi_{\cF}$ is a hypersurface invariant for the Schr\"odinger representation.

 \bibliography{ref-eqtns-of-kumm-type}
 \end{document}